\documentclass[12pt,a4paper]{article}
\usepackage[utf8]{inputenc}  
\usepackage{epic,eepic}
\usepackage{xcolor}
\usepackage{amsmath,amsfonts,amsmath}
\usepackage{bef_alex}
\numberwithin{equation}{section}
\usepackage{psfrag,bm} 

\def\FG{\bm}
\def\mafo{\mathrm}
\renewcommand{\log}{\ln}
\def\div{\mathop{\mafo{div}}}
\def\bflog{\mathop{\text{\bfseries log}}}

\def\REVER{\mathop{\rightleftharpoons}\limits}
\newcommand{\BW}{^\text{bw}}
\newcommand{\FW}{^\text{fw}}
\newcommand{\LB}{\lambda_\rmB}

\newcommand{\Proj}{\mathsf P}

\def\tot#1#2{\frac{\d #1}{\d #2}} 
\def\laplace{\Delta}
\def\grad{\nabla}
\def\ie{e}

\def\Norm#1{\left\| #1 \right\|}

\def\({\begin{eqnarray}}
\def\){\end{eqnarray}}
\def\[{\begin{eqnarray*}}
\def\]{\end{eqnarray*}}

\begin{document}

\pagestyle{plain}

\title{Decay to equilibrium for energy-reaction-diffusion systems}

\author{Jan Haskovec\thanks{King Abdullah University of Science and Technology, Thuwal 23955-6900, KSA}, Sabine Hittmeir\thanks{Faculty of Mathematics, University of Vienna, Oskar-Morgenstern-Platz 1, A-1090 Vienna, Austria},\\Peter Markowich\thanks{King Abdullah University of Science and Technology, Thuwal 23955-6900, KSA}{\ } and Alexander Mielke\thanks{Weierstra\ss{}-Institut f\"ur Angewandte Analysis und  Stochastik, Mohrenstra\ss{}e 39, 10117 Berlin, Germany}}  


\maketitle

\begin{abstract} We derive thermodynamically consistent models
  of reaction-diffusion equations coupled to a heat equation. While
  the total energy is conserved, the total entropy serves as a driving
  functional such that the full coupled system is a gradient flow. 
The novelty of the approach is the Onsager structure, which is the
dual form of a gradient system, and the formulation in terms of the
densities and the internal energy.  In these variables it is possible
to assume that the entropy density is strictly concave such that there
is a unique maximizer (thermodynamical equilibrium) given linear constraints
on the total energy and suitable density constraints.

We consider two particular systems of this type, namely, a
diffusion-reaction bipolar energy transport system, and a
drift-diffusion-reaction energy transport system with confining
potential. We prove corresponding entropy-entropy production
inequalities with explicitely calculable constants and establish the
convergence to thermodynamical equilibrium, at first in entropy and
further in $L^1$ using Czisz\'ar-Kullback-Pinsker type inequalities.
\end{abstract}

MSC: 



\section{Introduction}
\label{s:Intro}
The idea of this paper is to formulate the coupling between a
reaction-drift-diffusion system and a heat equation in terms of a
gradient flow system.  For reaction-diffusion systems a full gradient
structure was established only recently in \cite{Miel11GSRD},
including the non-isothermal case with a heat equation. The latter work
was largely inspired by the modeling paper \cite{AlGaHu02TDEM} and the
abstract theory on metric gradient flows, see e.g.\
\cite{Otto01GDEE,AmGiSa05GFMS,LieMie13GSGC}. We also refer to
\cite{Miel13TMER} for more details on Allen-Cahn or Cahn-Hilliard type
systems coupled to heat equations.  However, the coupling of different
gradient systems is non-trivial and the main observation of the latter
paper is that the coupling is largely simplified if we consider the
dual formulation, where the inverse $\bbK$ of the Riemannian metric
tensor $\bbG$ is used. We call the symmetric and positive
(semi)definite operator $\bbK$ an \emph{Onsager operator}, and the
triple $(\bfX,\Phi,\bbK)$ is called an \emph{Onsager system}, where
the state space $\bfX$ is a convex subset of a Banach space and
$\Phi:\bfX\to \R \cup\{\infty\}$ is the functional generating the evolutionary system
\[
\dot u = -\bbK(u) \rmD\Phi(u) &\quad \Longleftrightarrow \quad &
\bbG(u)\dot u= - \rmD \Phi(u).
\] 
We call the triple $(\bfX,\Phi,\bbG)$ the associated gradient
system. 

A major advantage of the Onsager form is its flexibility in modeling. 
This is due to the fact that the Onsager operator can be decomposed
into additive parts that account for different physical
phenomena - in our case, diffusion, heat transfer, and reaction.
In most applications the Onsager operator for non-isothermal systems 
has a special structure (cf.\ \cite{Edwa98ASDG,Otti05BET,Miel11FTDM})
with the free entropy being the correct driving potential
for the non-temperature part of the system, see
\eqref{eq:GO-y-r}.

We use the Onsager structure to prove convergence to equilibrium by an
application of the entropy method in two particular systems with
semiconductor-type reaction inspired by the Read-Shockley-Hall term,
see \cite{MarRinSchm90}.  First, we study a diffusion-reaction bipolar
energy transport system, and second, a drift-diffusion-reaction energy
transport system with confining potential.  In particular, we prove
entropy entropy-production inequalities with explicitly calculable
constants, using a generalization of the approach of
\cite{DiF-Fel-Mar,MiHaMa15UDER}, where the isothermal
reaction-diffusion system was 
considered.  This entails convergence to an entropy minimizing
equilibrium state, at first in entropy and further in $L^1$ norm using
Czisz\'ar-Kullback-Pinsker type inequalities, see e.g.\ \cite{UAMT00}.
The entropy approach is per se a nonlinear method avoiding any kind of
linearization and capable of providing explicitly computable
convergence rates. Moreover, being based on functional inequalities
rather than particular differential equations, it has the advantage of
being quite robust with respect to model variations.

To introduce the main ideas and notations we consider a simplified but still
nontrivial example, namely the reaction diffusion system
\begin{subequations}
  \label{eq:intro1}
\begin{align}
\dot u&= \delta \Delta u + \kappa \,\big(e^\alpha - u^\beta\big) ,\\
\dot e& = \delta \Delta e , 
\end{align}
\end{subequations}
where $u(t,x)>0$ is the density of a chemical species $X_u$ and $e(t,x)>0$
is the internal energy. The chemical species can be absorbed by or
generated from the background according to the mass action law $\beta
X_u \REVER \emptyset$, where the equilibrium state $w$ depends on the
internal energy, namely  $w(e)=e^{\alpha/\beta}$. This system can be
written as a gradient flow for the entropy functional $\calS$ and
the total conserved energy (assuming no-flux boundary conditions):
\begin{align*}
&\calS(u,e)=\int_\Omega S(u(x),e(x))\dd x \quad \text{ and } \quad 
\calE(u,e)=\int_\Omega e(x) \dd x, \text{ with }\\
&S(u,e)=s(e)-w(e)\LB\big( \frac{u}{w(e)} \big)= \wh s(e) - \LB(u) +
u\log w(e),
\end{align*}
where $\LB(\nu):= \nu \log \nu - \nu +1\geq 0$ with $\LB(1)=0$ and 
$\wh s(e)= s(e)- w(e)+1 $. Assuming that $\wh s$ is strictly
increasing and that $\wh s$ and $w$ are concave (i.e.\ $\alpha \leq
\beta$), implies that $S:{[0,\infty[}^2\to \R$ is concave, which is expected
from thermodynamical models. Moreover, the temperature $\theta
=1/\pl_e S(u,e)$ is positive. 

The gradient structure follows from the fact that the Onsager system
$(\bfX,\Phi,\bbK)$ generates \eqref{eq:intro1} if we choose 
\begin{align*}
&\bbK(u,e) \binom{\eta}{\eps} := - \div\Big(\bbM(u,e) \nabla \binom{\eta}{\eps} \Big) +
\bma{cc}\bbH(u,e)&0\\0 & 0\ema \binom{\eta}{\eps}, \\
&\text{where } \bbM(u,e)=  \delta
\big({-}\rmD^2S(u,e)\big)^{-1} \ \text{ and } \ \bbH(u,e)=\kappa \,
\frac{w(e)^\beta - u^\beta}{ \log w(e)- \log u}\:>0.
\end{align*}
With this choice we easily see that \eqref{eq:intro1} takes the form 
\[
\frac{\rmd}{\rmd t} \binom{u}{e} =  \bbK(u,e)\rmD \calS(u,e) \qquad
\text{and that } \bbK(u,e)\rmD\calE(u,e)=0,
\]
where the last relation leads to energy conservation, i.e. every
solution $t\mapsto (u(t),e(t))$ satisfies
$\calE(u(t),e(t))=\calE(u(0),e(0))=:E_0$.

Because of the strict concavity of $\calS$ , we expect that the
solutions $(u(t),e(t))$ converge to the unique maximizer of $\calS$
under the constraint $\calE(u,e)=E_0$, namely $(u_*,e_*)$
where $e_*=E_0/\text{vol}(\Omega)$ and $u_*=w(e_*)$. Thus, we can
define the non-negative and convex relative entropy 
\[
\calH(u,e)=\calS(u_*,e_*) + \rmD_e\calS(u_*,e_*)[e{-}e_*] -\calS(u,e),
\]
which satisfies $\calH(u,e)\geq \calH(u_*,e_*)=0$. 

The convergence to
equilibrium is now controlled by the entropy production $\calP$
defined as follows, where we reduce to the case $\alpha=1<\beta$ and $\wh
s(e)=\sigma e^{1/\beta}$ with $\sigma>0$ for notational simplicity:
\begin{align*}
&{-}\frac{\rmd}{\rmd t} \calH(u(t),e(t)) =  \calP(u(t),e(t))=
 \delta \calP_\text{diff}(u(t),e(t))+ \kappa
 \calP_\text{react}(u(t),e(t)), \quad\\ 
&\text{where } \calP_\text{diff}(u,e)= \int_\Omega \Big\{
\frac{\beta{-}1}\beta \,\frac{|\nabla u|^2}{u} + \frac{4\, e}\beta\,
 \Big| \nabla \sqrt{\tfrac{\ds u}{\ds e}} \Big|^2 + \sigma
\frac{\beta{-}1}{\beta^2} \,\frac{|\nabla e|^2}{e^{2-1/\beta}} \Big\}
\dd x \\
&\text{and } \calP_\text{react}(u,e)= \int_\Omega \bbH(u,e)\big(
\log(u/w(e))\big)^2 \dd x = \int_\Omega \frac1\beta\big(
u^\beta-e\big)\big( \log u^\beta - \log e\big) \dd x .
\end{align*}
After this modeling steps are done, the main task is to derive an entropy
entropy-production estimate in the form 
\begin{equation}
  \label{eq:I-EEPE}
\calP(u,e) \geq K(e_*)\calH(u,e) \text{ for all }(u,e)\text{ with }
\calE(u,e)=e_*\,\text{vol}(\Omega),
\end{equation}
which then leads, via $\frac{\rmd}{\rmd t}\calH(u,e)\leq -
K(e_*)\calH(u,e)$, to the exponential decay estimate 
\[
\calH(u(t),e(t)) \leq \exp\big({-}K(e_*)t\big)\, \calH(u(0),e(0))
\text{ for } t>0, 
\]
where $e_*=\calE(u(0),e(0))/\text{vol}(\Omega)$. 

In fact, we are not able to prove \eqref{eq:I-EEPE} in the general
form given here, but refer to \cite{MiHaMa15UDER} for such results in
the isothermal case. As in this work the general strategy is (i) to exploit
$\calP_\text{diff} $ to estimate the distance between $(u,e)$ and its
averages $(\ol u,e_*)$ and (ii) to exploit $\calP_\text{react} $ to
estimate $\ol u  - u_* $. For (i), we see that the first
and third term in $\calP_\text{diff}$ allow a first estimate, but
because of the special structure of $S$ which couples $u$ and $e$
non-trivially, we also need the second term, which gives a log-Sobolev
estimate for $u/e$ with respect to the measure $\frac{\ds e}{\ds e_*} \dd x$
such that we have to impose bounds of the form $0< \underline e \leq e(t,x) \leq
\ol e <\infty$ for showing \eqref{eq:I-EEPE}, see Propositions
\ref{prop:EEP} and \ref{prop:EEP2}.

The structure of the paper is as follows: In Section
\ref{s:GradientOnsager} we provide a general review on gradient and
Onsager systems and introduce non-isothermal systems.  In
Section \ref{s:RDS} we present the Onsager structure for a wide class
of isothermal reaction-diffusion systems based on the assumption of
detailed balance for the reaction system, cf.\ \cite{Miel11GSRD}. 
Reaction and diffusion can be discussed as separate dissipative
processes giving
$\bbK=\bbK_\text{diff}+\bbK_\text{react}$.  
In Section \ref{s:RDSTemp} we follow \cite{Miel13TMER} for the
modeling of non-isothermal reaction-diffusion systems and provide the
corresponding Onsager operator. As in example \eqref{eq:intro1}
we will see that it is advantageous to use the internal energy $e$ as
variable instead of the more common temperature $\theta=1/\pl_e
S(\bfu,e)$. This is even more evident in the analysis in Sections
\ref{s:Special} and \ref{s:Convergence}. 
In Section \ref{ss:CompSCBoltz} we also compare the general form of
these systems with the energy-transport models derived in
\cite{Juengel09,Juengel10} from a diffusive
scaling of the semiconductor Boltzmann equation.  In Section \ref{s:Special} we
derive special instances of the general system, in particular, a
diffusion-reaction bipolar energy-transport system, and a
drift-diffusion-reaction energy transport system with confining
potential $V(x)$.  In Section \ref{s:Convergence}, we apply the entropy
entropy-production method to study the convergence to equilibrium
for the two systems, where the focus is to derive the estimate 
\eqref{eq:I-EEPE}. The first model is posed on a $d$-dimensional
torus of homogeneous material, i.e.\ the constitutive functions
$w$, $s$, and hence $S$ are independent of $x$. The 
second model is considered in the full space setting, where
$w_i(x,e)=C_i \sqrt e\,\exp(-V(x))$ and $\wh s(x,e)=c \sqrt e
\,\exp(-V(x))$. Now the entropy entropy-production estimate
\eqref{eq:I-EEPE} can be derived in suitably weighted spaces, see
Proposition \ref{prop:EEP2}.  Finally, in
Appendix \ref{s:App} we provide an overview of auxiliary
results that we use throughout the paper.

\section{Gradient systems including heat equations} 
\label{s:GradientOnsager}

In this section we discuss some general background about
gradient systems and address the general question how the
temperature or other thermodynamic variables such as the internal
energy $e$ or the entropy $s$ can be included. We follow the ideas
developed in \cite{Miel11FTDM,Miel13TMER}. Instead of the metric
tensor $\bbG$ which is in the origin of the name gradient system, we
will use its inverse $\bbK=\bbG^{-1}$ which we call Onsager operator,
as it was Onsager's fundamental contribution in \cite{Onsa31RRIP} to
show that the matrix or operator $\bbK$ that maps thermodynamic
driving forces into rates should be symmetric (called reciprocal
relation at that time) and positive semidefinite.  Throughout our
arguments are formal and assume sufficient smoothness of the
potentials as well as the solutions (which is the common approach in
thermomechanics).

\subsection{General modeling with gradient systems}
\label{ss:Gradient}
A gradient system is a triple $(\bfX,\Phi,\bbG)$ where $\bfX$ is the
state space containing the states $U\in \bfX$. For simplicity we
assume that $\bfX$ is a reflexive Banach space with dual $\bfX^*$. 
The driving functional $\Phi:\bfX\to \R\cup\{\infty\}$ is
assumed to be differentiable (in a suitable way) such that the
potential restoring force is given by $-\rmD \Phi(U)\in \bfX^*$.
The third ingredient is a metric tensor $\bbG$, i.e.\ $\bbG(U):\bfX \to
\bfX^*$ is linear, symmetric and positive (semi-)definite. Indeed,
in a proper manifold setting, $\bbG$ maps the tangent space
$\rmT_U\bfX$ into the cotangent space
$\rmT_U^*\bfX = (\rmT_U\bfX)^*$.   
The gradient flow associated with $(\bfX,\Phi,\bbG)$ is the
(abstract) force balance
\begin{equation}
  \label{eq:GO1}
  \bbG(U)\dot U = - \rmD \Phi(U) \qquad \Longleftrightarrow \qquad 
\dot U= - \nabla_{\!\bbG} \Phi(U) =: - \bbK(U)\rmD \Phi(U), 
\end{equation}
where we recall that the ``gradient'' $\nabla_{\!\bbG} \Phi$ of the
functional $\Phi$ is an element of $\bfX$ (in contrast to the
differential $\rmD \Phi(U)\in \bfX^*$) and is calculated via 
$\bbK(U)\rmD\Phi(U)$ with $\bbK(U):=\bbG(U)^{-1}$. The left equation
in \eqref{eq:GO1} is an abstract force balance, since
$\bbG(U)\dot U \in \bfX^*$ can be seen as a viscous force arising from
the motion 
of $U$. The equation on the right-hand side is a rate equation,
where the equality is formulated in the tangent space $\bfX$. 

The  
symmetries of $\bbG$ and $\bbK$ allow us to define the
associated primal and dual \emph{dissipation potentials} 
$\Psi:\bfX\ti \bfX\to [0,\infty]$ and $\Psi^*:\bfX\ti \bfX^*\to
[0,\infty]$, respectively, via
\[  
\Psi(U,V)= \frac12\langle \bbG(U)V, V\rangle \quad \text{and}
\quad \Psi^*(U,\Xi)= \frac12\langle \Xi,\bbK(U)\Xi \rangle,
\]
where $\Psi^*(U,\cdot)$ is the Fenchel--Legendre transform of
$\Psi(U,\cdot)$. 
If $\Phi$ is the negative total entropy, then $\Psi$ is called the
\emph{entropy production potential}. 

Hence, using $\rmD_{V} \Psi(U,V)=\bbG(U)V$ and $\rmD_\Xi
\Psi^*(U,\Xi)=\bbK(U)\Xi$ the equations in \eqref{eq:GO1} can be
written as  
\[
0= \rmD_V\Psi(U,\dot U) + \rmD \Phi(U) \qquad \Longleftrightarrow \qquad 
\dot U = \rmD_{\Xi} \Psi^*(U,{-}\rmD \Phi(U)),
\]
which are also the correct forms for so-called generalized gradient systems, 
where $\Psi(U,\cdot)$ and
$\Psi^*(U,\cdot)$ are not quadratic, see
\cite{Miel14?EGCG,LMPR15?MOGG}. 

The importance of gradient systems is clearly motivated in the theory
of thermodynamics, namely by the \emph{Onsager symmetry principle},
see \cite{Onsa31RRIP,DegMaz84NET}.  Strictly speaking, this principle
is only derived for systems close to thermodynamic equilibrium, see 
\cite{Otti05BET} for physical justifications to use these principles
in a wider range. The symmetry principle has
two forms, both of which are important for reaction-diffusion
systems. (I) In the first case one considers a spatially homogeneous
system described by a state vector $z$, which is a small perturbation
of the equilibrium. Then, its macroscopic rate $\dot z$ is given in
the form $- \bbH\zeta$, where $\zeta= -\rmD S(z)$ is the
thermodynamically conjugate driving force and $S$ is the entropy.  The
symmetry relation states that the matrix $\bbH$ has to be symmetric,
while the entropy production principle $\frac{\rmd}{\rmd t} S(z(t))=
\rmD S(z)\cdot \bbH \rmD S(z)\geq 0$ implies that $\bbH$ has to be positive
semidefinite. (II) 
In the second case one considers a spatially extended system with
densities $u_i>0$ defining a vector
$\bfu=(u_i)_{i=1,...,I}:\Omega \to {]0,\infty[}^I$ and a total entropy
$\calS(\bfu)=\int_\Omega S(x,\bfu(x))\dd x$. If each total mass $m_i :=\int_\Omega
u_i(x)\dd x $ is conserved, then the densities 
satisfy a balance equation of the form 
\[
\dot\bfu + \div \bfj_\bfu=0 \quad \text{with} \quad  \bfj_\bfu = \bbM
\nabla \bfmu, 
\]
where the vector $\bfmu$ of the chemical potentials is given by 
$\bfmu = \rmD \calS(\bfu)$, i.e.\ $\mu_i(x)=\pl_{u_i} S(x,u(x))$. 
Again, the symmetry and entropy principle imply that $\bbM$
is a symmetric and positive semidefinite tensor (of fourth order), see
\cite{Onsa31RRIP}. 

Note that in this work we will call $\bfmu=\rmD \calS$ the
thermodynamic driving force (rather than a potential).
In our approach the driving force is lying in the dual space of the variable
(here $\bfu$), while $\nabla \mu_j$ relates to gradient in the
physical domain $\Omega$. However, more importantly, we will couple
the equation $\dot\bfu+\div \big(\bbM \nabla \bfmu\big) =0$ with
$\bfmu=\rmD S$ in the form $\dot \bfu= \bbK(\bfu)\rmD S$, where
$\bbK(u)=- \div\big( \bbM\nabla \Box\big)$ is a symmetric operator.

As was observed in \cite{Miel11GSRD,GliMie13GSSC,Miel13TMER},
it is advantageous to use the Onsager operator $\bbK$ for modeling
purposes rather than the metric operator $\bbG=\bbK^{-1}$, so we will
also write $(\bfX,\Phi,\bbK)$  for the gradient system and call it
Onsager system then. 
The flexibility in modeling arises from the fact that 
 evolutionary systems are often  written in rate form where the 
vector field is additively decomposed into different physical
phenomena.  This additive split can be also used for the Onsager
operator, as long as all the different effects are driven by the same
functional $\Phi$. Below we will see that $\bbK$ takes the additive form 
\[
\bbK = \bbK_\text{diff} + \bbK_\text{react}
+\bbK_\text{heat},
\]
such that the evolution equation reads 
\[
\dot U=-\Big(  \bbK_\text{diff}\rmD\Phi +
\bbK_\text{react} \rmD\Phi + \bbK_\text{heat}  \rmD\Phi \Big) = 
-\bbK \rmD\Phi.
\]
A similar additive split is not possible for the metric $\bbG$, as
the inverse operator to a sum of operators is difficult to express.

\subsection{Non-isothermal Onsager systems}
\label{ss:IsoNoniso}

In the isothermal case the functional $\Phi$ is typically the
free energy, and the state $U$ consists of positive densities $u_i$ of
phase indicators $\varphi_j$. In the non-isothermal case the
functional $\Phi$ is the total entropy $\calS$ and an additional
scalar thermal variable $r$ is needed, which can be the
absolute temperature $\theta>0$, the internal energy density $e$,
the entropy density $s$, or some variable derived from those. As in
\cite[Sect.\,2.3]{Miel11FTDM} we will
keep $r$ unspecified at this stage, because this elucidates the
general structure. Hence the states take the form 
$U=(\bfy,r)$, and we consider the functionals
\begin{equation}
\label{eq:2.calEcalS}
\calE(\bfy,r)=\int_\Omega E(x,\bfy(x),r(x))\dd x \quad
\text{and} \quad 
\calS(\bfy,r)=\int_\Omega S(x,\bfy(x),r(x))\dd x,
\end{equation}
where the constitutive functions $E$ and $S$ are related by  Gibbs formula
defining the temperature
\[
\theta= \Theta(x,\bfy,r):= \frac{\pl_r E(x,\bfy,r)} {\pl_r S(x,\bfy,r)}.
\]
 Without loss of generality, we subsequently assume that $\pl_r E$ and $\pl_r S$ are positive. 

We also argue that physically relevant driving forces should not
depend on the choice of $r\in \{\theta,e,s\}$. Thus, introducing the
Helmholtz free energy $\psi=e-\theta s$ and the Helmholtz free entropy
$\eta = -\psi/\theta= s -e/\theta$ (also called Massieu potential),
we have the formulas
\begin{align*}
&\psi=F(x,\bfy,r):=E(x,\bfy,r){-}\Theta(x,\bfy,r)S(x,\bfy,r)  
\text{ and } 
\\
&\eta=H(x,\bfy,r):=S(x,\bfy,r){-}\frac{E(x,\bfy,r)}{\Theta(x,\bfy,r)}.
\end{align*}
The point here is that the driving forces $\pl_\bfy F$ and
$\pl_\bfy S$ are independent of the choice of $r$ when the arguments
are transformed correspondingly. 

In the non-isothermal case the total entropy $\calS$ (with the
physically correct sign) is increasing, so $\Phi=-\calS$ is the
driving potential for the gradient flow. However, we will not make
this distinction in the text; instead, we will always use the
corresponding correct signs in the formulas.  Our Onsager system
$(\bfX,\calS,\bbK)$ hence gives rise to the equation $\dot U = \bbK(U)
\rmD \calS(U)$.  In order to have energy conservation, we need
\[
0 = \frac\rmd{\rmd t} \calE(U)=
\langle \rmD \calE(U),\dot U\rangle = \langle
\rmD\calE(U),\bbK(U) \rmD \calS(U)\rangle =
\langle \bbK(U)\rmD \calE(U), \rmD \calS(U)\rangle,
\]
where we used $\bbK=\bbK^*$.
Hence, it is sufficient (but not necessary) to impose the condition
\[
\bbK(U)\rmD \calE(U)= 0 \quad \text{for all } U \in \bfX. 
\]
 
In many applications the Onsager operator for non-isothermal systems 
has a special structure
(cf.\ \cite{Edwa98ASDG,Otti05BET,Miel11FTDM}), namely 
\begin{subequations}
\label{eq:2.SpecialK}
\begin{align}
&\bbK(\bfy,r) = \calM^*_\calE 
\bma{cc}\bbK^\text{simple}_\bfy &0\\0& \bbK^\text{simple}_\text{heat} \ema
\calM_\calE \quad \text{ with }\\
&\bbK^\text{simple}_\text{heat}\rho=-\div \big(k_\text{heat}(\bfy,r) 
\nabla \rho \big) \quad \text{and} \quad
\calM_\calE=\bma{cc}
I&-\frac1{\pl_r\calE}\rmD_\bfy \calE\\
0&\frac1{\pl_r\calE} \ema .
\end{align}
\end{subequations}
The definition of $\calM_\calE$ implies that  
\[
\calM_\calE\rmD \calE=\binom{0}1\quad \text{and} \quad 
\calM_\calE\rmD\calS=\binom{\rmD_\bfy\calS-\frac1\Theta
    \rmD_\bfy\calE}{1/\Theta} = 
\binom{\rmD_\bfy \mathfrak H(\bfy,r)}{1/\Theta},
\]
where $\mathfrak H(\bfy,r)=\int_\Omega H(x,\bfy(x),r(x)) \dd x$ is the total free
entropy. 
Since $\bbK^\text{simple}_\text{heat} 1\equiv 0$, we have the desired relation
$\bbK\rmD\calE\equiv 0$ for  energy conservation. 
Moreover, the coupled system $(\dot\bfy,\dot
r)=\bbK(\bfy,r)\rmD\calS(\bfy,r)$ can be rewritten in the form 
\begin{align}
  \label{eq:GO-y-r}
&\dot\bfy = \bbK^\text{simple}_\bfy(\bfy,r) \rmD_\bfy \mathfrak H(\bfy,r), \quad
\dot r = \frac1{\pl_r E} \Big( \pl_\bfy E \cdot \dot \bfy  + 
\bbK^\text{simple}_\text{heat}(\bfy,r)\big(1/\Theta(\bfy,r)\big) \Big) 
\end{align} 
We conclude that in the non-isothermal case with conserved
energy $\calE$ the correct driving potential for the non-temperature
part $\bfy$ of the system is the free entropy
$\mathfrak H(\bfy,\theta)$, see also \cite{Miel15?FEFE}. Nevertheless,
$\calS$ is the functional for the Onsager system, the correction
$-\frac1{\pl_r E}\pl_\bfy\calE$ arises because of energy
conservation, which is encoded in the way $\bbK$ is constructed from
$\bbK^\text{simple}_\bfy$ and $\bbK^\text{simple}_\text{heat}$. 

We end this section by observing that the case $r=e$, namely
$\wh E(x,\bfy,e):= e $, leads to an especially simple case, because
$\pl_e\wh E \equiv 1$ and $\pl_\bfy \wh E(x,\bfy,e)\equiv 0$. Indeed,
\eqref{eq:GO-y-r} turns into
\begin{align} 
  \label{eq:GO-y-e}
& \dot\bfy = \bbK^\text{simple}_\bfy(\bfy,e) \rmD_\bfy \calS(\bfy,e), \quad
\dot e = \bbK^\text{simple}_\text{heat}(\bfy,e)\Big(1/\wh\Theta(\bfy,e) \Big). 
\end{align}

\section{Isothermal reaction-diffusion systems}
\label{s:RDS}

While the gradient structure for scalar diffusion equations (e.g.\
porous medium equation or the Fokker-Planck equation) is well known
(cf.\ \cite{JoKiOt98VFFP,Otto01GDEE}), the Onsager structure for a
wider class of reaction-diffusion systems is less known. It was
established in a few particular cases (see \cite{Yong08ICPD,
  GrmOtt97DTCF12}), but only highlighted in its own right in
\cite{Miel11GSRD,GliMie13GSSC}. The central point is that in the
Onsager form we have an additive splitting of the Onsager operator
into a diffusive part and a reaction part, namely $\dot \bfu=
-\big(\bbK_\text{diff}(\bfu)+ \bbK_{\text{react}}(\bfu)\big)
\calF_\text{chem}(\bfu)$, where $\bfu:\Omega \to {]0,\infty[}^I $ is
the vector of densities of the species $X_1,...,X_I$.  The free-energy
functional $\calF_\text{chem}$, which is also called the relative
entropy with respect to the reference density $\bfu^*$, takes the form
\begin{equation}
  \label{eq:2FChem}
\calF_\text{chem}(\bfu)=\int_\Omega \sum_{i=1}^I
u_i^* \LB (u_i(x)/u_i^*) \dd x  \quad \text{ where} \quad 
\LB(\nu):=\nu \log \nu-\nu +1. 
\end{equation}
We will now discuss the diffusive and reactive parts separately.

\subsection{Diffusion systems} 
\label{sss:DiffSystems}

For the gradient structure of diffusion systems $\dot \bfu =
\div\big(\bbM(\bfu)\nabla \bfu\big)$ one might be tempted to use a
functional involving the gradient $\nabla u$, however we have to use
the relative entropy as a driving functional, because we have to use
the same functional for modeling the reactions. Hence, we use the
Wasserstein approach to diffusion introduced by Otto in
\cite{JoKiOt98VFFP,Otto01GDEE}.

The diffusion system will take the form 
$\dot \bfu= - \bbK_\text{diff}(\bfu) \rmD \calF_\text{chem}(u)$ with the
Onsager operator $\bbK_\text{diff}$ given via 
\[
\bbK_\text{diff}(\bfu)\bfmu= - \div \big( \wt \bbM(\bfu)\nabla
\bfmu\big),
\]
where $\wt \bbM(\bfu):\R^{m\ti d} \to\R^{m\ti d}$ is a symmetric and
positive semi-definite tensor of order 4. The Onsager operator can
also be implicitly defined via the dual dissipation potential, which
will be useful later:
\[
\Psi^*_\text{Wass}(\bfu,\bfmu) = \frac12 \int_\Omega
\nabla\bfmu{:}\wt\bbM(\bfu)\nabla \bfmu \dd x,
\]
where $\bfmu=(\mu_i)_{i=1,..,I} $ is the vector of chemical
potentials, which occurs as the driving force 
\[
\bfmu=\rmD_\bfu \calF_\text{chem}(\bfu)=
\bflog\bfu- \bflog\bfu^*.
\]
Hence, if the reference densities $\bfu^*$ are spatially constant
(which is, however, usually not true in heterostructures like
semiconductors), the Onsager system leads to the diffusion system
\[
\dot\bfu= \div\big(\wt\bbM(\bfu)
\nabla(\bflog\bfu{-}\bflog \bfu^*)\big) =
\div\big( \bbM(\bfu) \nabla \bfu\big),
\ \text{ where } \wt\bbM(\bfu)= \bbM(\bfu)\mafo{diag}(\bfu).
\]

\subsection{Chemical reaction kinetics} 
\label{sss:ChemReact}
Chemical reaction systems are ODE systems of the type $\dot\bfu=\bfR(\bfu)$,
where often the right-hand side is written in terms of polynomials
associated to the reaction kinetics. It was observed in
\cite{Miel11GSRD} that under the assumption of detailed balance (also
called reversibility) such system have a gradient structure with the
relative entropy as the driving functional. 
We assume that there are $R$ reactions of mass-action type (cf.\ e.g.\ 
\cite{DegMaz84NET,GliMie13GSSC}) between the species $X_1,...,X_I$ in
the form
\[
\alpha^r_1 X_1+\cdots +\alpha^r_I X_I \ \REVER^{k_r\FW}_{k_r\BW} \  
\beta^r_1 X_1+\cdots +\beta^r_I X_I, \qquad r=1,\dots,R,
\]
where $k_r\BW$ and $k_r\FW$ are the backward and forward reaction
rates, and the vectors $\bfalpha^r,\;
\bfbeta^r\in \N_0^I$ contain the stoichiometric coefficients.
For instance, for the chemical reaction 2\,CO + 1\,O$_2 \ \REVER \ $2\,CO$_2$ we
have $\bfalpha=(2,1,0)^\top$ and $\bfbeta=(0,0,2)^\top$. 

The associated reaction system for the densities (in a spatially
homogeneous system, where diffusion can be neglected) reads 
\begin{equation}
  \label{eq:ReactSyst}
   \dot\bfu =\bfR(\bfu):=-\sum_{r=1}^R \big(k_r\FW\bfu^{\bfalpha^r}{-}
k_r\BW\bfu^{\bfbeta^r}\big)\Big( \bfalpha^r- \bfbeta^r\Big),
\end{equation}
where we use the monomial notation $\bfu^\bfalpha=
u_1^{\alpha_1}\cdots u_I^{\alpha_I}$.  The main assumption to obtain a
gradient structure is that of \emph{detailed balance}, which means
that there exists a reference density vector $\bfu^*$ such that all
$R$ reactions are balanced individually, namely
\begin{equation}
  \label{eq:DetBalance}
\exists\,\bfu_*  \in \left]0,\infty\right[^I\ \forall\,
r=1,...,R\ \forall \, \bfu\in \left]0,\infty\right[^I : \ k_r\FW(\bfu)
\bfu_*^{\bfalpha^r} = k_r\BW(\bfu)  \bfu_*^{\bfbeta^r} =:k^*_r(\bfu). 
\end{equation}
Here we used the freedom to let reaction coefficients
depend on the densities (and later also on other material
properties like temperature). 

We now define the Onsager matrix 
\begin{equation}
  \label{eq:2OnsagerMatrix}
  \bbH(\bfu)=\sum_{r=1}^R k^*_r (\bfu)
  \bfLambda\big(\tfrac{\bfu^{\bfalpha^r}}{\bfu_*^{\bfalpha^r}}  ,
       \tfrac{\bfu^{\bfbeta^r}}{\bfu_*^{\bfbeta^r}} \big) 
  \big(\bfalpha^r{-}\bfbeta^r\big) 
  {\otimes} \big(\bfalpha^r{-}\bfbeta^r\big) \text{ with }
\bfLambda(a,b)=\frac{a-b}{\log a - \log b} 
\end{equation}
and find, following  \cite{Miel11GSRD}, that the reaction system
\eqref{eq:ReactSyst} takes the form  
\[
\dot \bfu= \bfR(\bfu)= - \bbH(\bfu) \rmD\calF_\text{chem}(\bfu)
\]
with $\calF_\text{chem}$ given by \eqref{eq:2FChem}.
This follows easily by using the definition of $\bfLambda$ and the
rules for logarithms, namely
$\big(\bfalpha^r{-}\bfbeta^r\big)\cdot \big(\bfmu{-}\bfmu^*)= 
\log\!\big({\bfu^{\bfalpha^r}}/{\bfu_*^{\bfalpha^r}}\big)- 
\log\!\big({\bfu^{\bfbeta^r}}/{\bfu_*^{\bfbeta^r}}\big) $. 
The quotient $\Lambda(a,b)=\frac{a-b}{\log a- \log b}$ (or variants of
it) have occurred occasionally in the modeling of reaction
kinetics, see \cite[Sect. V.B]{GrmOtt97DTCF12},
\cite{EdeGil08TCKM,Eder09TMMR}, and \cite[Sect.\,7]{Yong08ICPD}.

We refer to \cite{LMPR15?MOGG,MPPR15?LDPC} for a different
gradient structure, where $\calF_\text{chem}$ is the same, but the
quadratic dual dissipation potential $\Psi^*_\text{quad}(\bfu,\bfmu)=\frac12
\bfmu\cdot \bbH(\bfu)\bfmu$ is replaced by a non-quadratic one that is
derived from a large-deviation principle. 
  
\subsection{Coupling diffusion and reaction}
\label{sss:CouplDiffReact}

We summarize the previous two subsections by stating the following
result derived in \cite{Miel11GSRD}. 

\begin{theorem}
  \label{th:RDS} 
If the reaction-diffusion system 
\begin{equation}
  \label{eq:RDSthm}
\dot\bfu = \div\!\big( \bbM(\bfu) \nabla \bfu\big) + \bfR(\bfu)
\end{equation}
with $\bfR(\bfu)=-\sum_{r=1}^R \big(k_r\FW(\bfu)\bfu^{\bfalpha^r}{-}
k_r\BW(\bfu)\bfu^{\bfbeta^r}\big)\big( \bfalpha^r {-} \bfbeta^r\big)$ satisfies
the detailed balance condition \eqref{eq:DetBalance} and 
$\wt\bbM(\bfu)=\bbM(\bfu)\mafo{diag}(\bfu)$ is symmetric and positive
semidefinite, then it is an Onsager system $\dot \bfu = -
\bbK_{\mafo{RD}}(\bfu) \rmD\calF_\mafo{chem}(\bfu)$ with
\[
\calF_\mafo{chem}(\bfu)= \!\!\int_\Omega {\ts\sum\limits_{i=1}^I }u_i^* 
\LB (u_i(x) /u_i^*) \dd x  \text{ and } 
 \bfPsi_\mafo{RD}^*(\bfu,\bfmu)=\frac12\int_\Omega\!\!
\nabla\bfmu{:}\wt\bbM(\bfu){:}\nabla \bfmu {+} \bfmu {\cdot}
\bbH(\bfu) {\cdot} \bfmu \dd x .
\]
\end{theorem}

We mention that many reaction-diffusion systems studied (including
semiconductor models involving an elliptic equation for the
electrostatic potential) have the gradient structure developed above,
see,
e.g.\ \cite{GliHun05GERP,DesFel06EDTE,DesFel07EMRD,Glit09EEER,BotPie11ILRD}.
So far, the gradient structure was not used explicitly, only the
Liapunov property of the free energy (relative entropy) was exploited.

As we assume throughout that our system \eqref{eq:RDSthm}
has no-flux boundary conditions, we may have additional
conservation laws, often called conservation of atomic mass. With
\[
\bbS:= \mathrm{span}\{\:\bfalpha^r{-}\bfbeta^r \:|\: r=1,\ldots,R\:\}
\subset \R^I  \ \text{ and } \ 
\bbS^\bot:= \set{ \bfp \in \R^I }{ \bfp\cdot \bfgamma = 0 \text{ for
    all } \bfgamma \in \bbS },
\]
we denote the stoichiometric subspace $\bbS$ associated to the reaction system
\eqref{eq:ReactSyst} and its orthogonal complement. From $\bbH(\bfu)\bfp
\equiv 0$ for all $\bfp\in \bbS^\bot$ we conclude that the 
functionals  
$\calC_\bfp(\bfu):= \int_\Omega \bfp\cdot \bfu(x) \dd x$
are conserved along solutions of the reaction-diffusion system. 
Defining by $\Proj :\R^I\to \R^I$ the orthogonal projection onto
$\bbS^\bot$, we obtain that 
\begin{equation}
  \label{eq:calP--bbS-bot}
\calC_{\bbS^\bot}(\bfu):= \int_\Omega \Proj  \bfu(x) \dd x \ \in
\bbS^\bot
\end{equation}
is conserved along solutions of \eqref{eq:RDSthm}.

\section{Non-isothermal reaction-diffusion systems}
\label{s:RDSTemp}

We now restrict ourselves to a system described by $(\bfu,r)$
with the total energy and entropy functionals 
\[
\calE(\bfu,r)=\int_\Omega E(x,\bfu(x),r(x)) \dd x \quad
\text{and} \quad 
\calS(\bfu,r)=\int_\Omega S(x,\bfu(x),r(x)) \dd x,
\]
where the integrands are strictly local, i.e.\ they do not depend on
$\nabla \bfu$ and $\nabla r$. Recall that $r$ is a scalar
thermodynamical variable such as $\theta$, $e$, or $s$. Throughout the
paper, the densities may explicitly depend on the material point, but
we will omit this dependence in the sequel. The energy density $E$
and the entropy density $S$ satisfy the Gibbs relation $\pl_r E(\bfu,
r)=\Theta(\bfu, r) \pl_r S(\bfu, r)$ and the positivity of the
specific heat $\pl_\theta E>0$.

The dual entropy production potential $\Psi^*$ will depend on the
state $(\bfu,r)$ and the thermodynamic conjugate variables
$(\bfmu,\rho)$. In principle, $\Psi^*$ will contain three
parts, namely a diffusion part, a reaction part, and a part for heat
conduction. However, the heat conduction and the diffusion can be
joined into one quadratic form on $(\nabla\bfmu,\nabla\rho)$, thus
allowing for ``cross-diffusion'' effects between chemical diffusion
and heat transfer, which is needed to model thermophilic or
thermophobic materials occurring for instance in polymers, see e.g.\
\cite{AndWei12TBB}.

To guarantee energy conservation, we follow
\cite[Sect.\,3.6]{Miel11GSRD} in using the special form
\eqref{eq:2.SpecialK} and consider 
\begin{subequations}
  \label{eq:2PsiNonisoRDS}
\begin{align}
&\Psi^*(\bfu,r;\bfmu,\rho)=\wt\Psi^*(\bfu,r;
\bfmu{-}\tfrac{\rho}{\pl_r E} \pl_\bfu E, \tfrac{\rho}
{\pl_r E})\\
&\text{with }\wt\Psi^*(\bfu,r;\wt\bfmu,\wt\rho)= 
 \frac12\int_\Omega (
\nabla\wt\bfmu,\nabla\wt\rho){\mdot} \wt\bbM(\bfu,r)
( \nabla\wt\bfmu,\nabla\wt\rho)+ \wt\bfmu\vdot \bbH(\bfu,r)\wt\bfmu\dd x,
\end{align}
\end{subequations}
where $\bbH$ is given as in \eqref{eq:2OnsagerMatrix}. The mobility
tensor $\bbM(\bfu,r):\R^{I\ti d}\ti \R^d\to \R^{I\ti d}\ti \R^d$
is symmetric and positive semidefinite and has the block structure
\[
\bbM(\bfu,r) = \bma{cc} \bbM_{\bfu\bfu}(\bfu,r)& 
\bbM_{\bfu r}(\bfu,r)\\ 
\bbM^*_{\bfu r}(\bfu,r) & \bbM_{r r}(\bfu,r)
\ema. 
\]
The associated Onsager operator $\bbK$ satisfies $\bbK \rmD
\calE\equiv 0$, and  we see that $\Psi^*$ only depends on 
\[
\rmD_\bfu \calS - \frac1\theta \rmD_\bfu \calE =
\pl_\bfu S -\frac1\Theta \pl_\bfu E = \pl_\bfu H \quad \text{and}
\quad \frac{\pl_\theta S}{ \pl_\theta E}= \frac1\Theta,
\]
where $H=-\psi/\theta = -F(\bfu,\Theta)/\Theta= S-E/\Theta$ is the
free entropy.

\subsection{Reaction-diffusion systems with temperature}
\label{ss:RDStemperature}

For completeness and for the comparison with our modeling based
on the internal energy, we also consider the choice 
$r=\theta$, which leads to the following system for  $(\bfu,\theta)$:
\begin{align*}
\dot \bfu& \ts = -\div \bfj_\bfu  + \bbH(\bfu,\theta)\big(\pl_\bfu
\ol S(\bfu,\theta) {-}\frac1\theta \pl_\bfu \ol E(\bfu,\theta)\big),\\
\dot \theta& \ts = -\frac1{\pl_\theta \ol E} \div \bfj_\theta +
\frac1{\pl_\theta \ol E}  \pl_\bfu \ol E {\vdot} 
\big( \div \bfj_\bfu - \bbH(\bfu,\theta)\big(\pl_\bfu
\ol S(\bfu,\theta) {-}\frac1\theta \pl_\bfu \ol E(\bfu,\theta) \big),
\end{align*}
with the fluxes
\begin{align*}
\bfj_\bfu& \ts = \bbM_{\bfu\bfu}(\bfu,\theta) \nabla\big( \pl_\bfu
\ol S(\bfu,\theta){-}\frac1\theta \pl_\bfu \ol E(\bfu,\theta) \big)
+ \bbM_{\bfu\theta}(\bfu,\theta)\nabla(1/\theta), \\
\bfj_\theta& \ts = \bbM^*_{\bfu\theta}(\bfu,\theta) \nabla\big( \pl_\bfu
\ol S(\bfu,\theta){-}\frac1\theta \pl_\bfu \ol E(\bfu,\theta) \big)+
\bbM_{\theta\theta}(\bfu,\theta) \nabla(1/\theta).
\end{align*}
In the examples below, we will see that the formulation in terms
of $(\bfu,e)$ gives a much simpler system. Moreover, in
general the function $s=\ol S(\bfu,\theta)$ does not enjoy any
concavity properties, in contrast to the function $s=\wh S(\bfu,e)$.

\subsection{Reaction-diffusion systems with internal energy}
\label{sss:RDSIntern}

According to \cite{AlGaHu02TDEM,Miel11GSRD}, it is more convenient to
study energy-reaction-diffusion systems with the choice $r=e$, for the
following two reasons. First, it is physically reasonable to assume
that 
\[
s= \wh S(x,\bfu,e)
\]
is a strictly concave function in the extensive variables $\bfu$ and
$e$. Second, 
\[
\wh\calE(\bfu,e)= \int_\Omega e(x) \dd x
\]
provides a linear functional to express energy conservation.
Thus, the energy and entropy functionals are  
\[
\wh\calE(\bfu,e) = \int_\Omega e(x) \dd x \quad \text{and} \quad
\wh\calS(\bfu,e)=\int_\Omega \wh S(x, \bfu(x),e(x)) \dd x.
\]   
Now the Gibbs relation leads to the definition of temperature as 
\[
\theta =\wh\Theta(x,\bfu,e):= 1/\pl_e \wh S(x,\bfu,e),
\]
where the relation $\pl_e \wh S(\bfu,e)>0$ is imposed. 

The driving force through the
free entropy is most simple, as $\pl_\bfu H = \pl_\bfu \wh S$, since
using $\wh E(\bfu,e):=e$ we have $\pl_\bfu \wh E\equiv 0$. 
Thus, the dual entropy-production potential  in terms of $(\bfu,e)$
takes the form
\begin{equation}\label{ent.prod}
\wh\Psi^*(\bfu,e; \bfmu,\eps) =\frac12 \int_\Omega 
(\nabla\bfmu,\nabla \eps) {\mdot} \wh \bbM(\bfu,e)
(\nabla\bfmu,\nabla \eps) + \bfmu{\vdot} \wh\bbH(\bfu,e) \bfmu\, \dd x,
\end{equation}
where $\wh\bbM$ and $\wh\bbH$ are positive semidefinite. As a
consequence of the simple form of $\wh\calE$, and hence of $\wh\Psi^*$,
the evolution equations for $(\bfu,e)$ take the simple form
\begin{equation}
  \label{eq:RDS-u-e}
\binom{\dot \bfu}{\dot e} =- \div\Big(
\wh\bbM(\bfu,e) \nabla\binom{\pl_\bfu \wh S(\bfu,e)}{\pl_e \wh S(\bfu,e) }\Big)
   +\binom{ \wh\bbH(\bfu,e)\pl_\bfu S(\bfu,e)}{0} .
\end{equation}

This form has the major advantage that we easily see the parabolic
nature. Moreover, there are no source terms in the energy equation. 

\subsection{Two examples}
\label{ss:Examples}

The major modeling issue in non-isothermal reaction-diffusion
system is the fact that the thermodynamical equilibrium satisfying
the detailed-balance condition \eqref{eq:DetBalance} should depend on
the temperature. In our modeling we will make it dependent on the
internal energy $e$ and write $\bfu=\bfw(x,e)=(w_1(x,e),...,w_I(x,e)$ for
the equilibrium. Obviously, this equilibrium is obtained by maximizing
$\wh S(x,\bfu,e)$ with respect to $\bfu$. Note that we always assume
that $\wh S(x,\cdot,\cdot)$ is strictly concave, so a maximizer is
unique. 

We give two examples of relevant entropy functions $\wh S$. The first one
depends on $e$ in the form $\log e$ like for gases, while in the
second one can choose $\wh S$ such that it stays finite at $e=0$ which
corresponds to $\theta=0$, which is suitable for modeling solids,
i.e.\ electrochemical species diffusing and reacting inside a solid. 
Both examples use an additive form of the entropy, which is a good
model only in the case of low densities; of course also entropies with
strong coupling between the components $u_1,...,u_I$ could be considered.

\textbf{Example 1} has the form 
\[
\wh S_1(x,\bfu,e)= \sum_{i=1}^I \Big( c_i u_i \log e - u_{*,i}
\LB\big(u_i/u_{*,i}\big)\Big),
\]
where $c_i\geq 0$ and $u_{*,i}>0$ may depend on $x\in \Omega$. 
Using $\LB'(\nu)=\log \nu$ we find that $\pl_\bfu \wh S_1(u,e)=0$ for 
\[
\bfu = \bfw(e)= \big( u_{*,i} e^{c_i} \big)_{i=1,...,I}.
\]
Gibbs relation gives the relation between internal energy
and temperature in the form
\[
\frac{1}{\theta}= \frac{1}{\wh\Theta(\bfu,e)} =\frac{ \bfc\cdot \bfu}{e} \quad \text{ or }
\quad e = \ol E(\bfu,\theta)=\theta\: 
\bfc\cdot \bfu, \ \text{ where }\bfc=(c_1,...,c_I).
\]
For the second derivative of $\wh S_1$ we have the relation 
\begin{align*}
-\tbinom{\bfmu}\eps \cdot \pl^2\wh S_1(\bfu,e)\tbinom{\bfmu}\eps 
&= \sum_{i=1}^I\frac{\mu_i^2}{u_i} - 
  2\frac{\bfc\cdot \bfmu}{e} \: \eps +\frac{\bfc\cdot\bfu}{e^2}\: \eps^2\\
&\overset{(1)}{\geq} \sum_{i=1}^I \frac{\mu_i^2}{u_i} - \frac{(\bfc \cdot
  \bfmu)^2}{\bfc \cdot \bfu} 
\ \overset{(2)}{\geq} \ \sum_{i=1}^I \frac{\mu_i^2}{u_i} - \sum_{i=1}^I \frac{c_i
  \mu_i^2}{u_i},
\end{align*}
where we minimized over $\eps$ to obtain $\overset{(1)}{\geq}$, while
$\overset{(2)}{\geq}$ follows using the Cauchy-Schwarz estimate. Thus, we
see that concavity of $\wh S_1$ holds if and only if all $c_i$ lie in
$[0,1]$. 

However, inserting $e=\ol E(\bfu,\theta)$ into $\wh S_1$, we see that 
$\ol S_1(\bfu,\theta)=\wh S_1(\bfu,\ol E(\bfu,\theta))$ is not concave in
general. Indeed, in the case $I=1$ with $u_*=1$ we have
\[
\ol S(u,\theta)= c u \log \theta -(1{-}c)u\log u + (1{+}c\log c)u -1
\]
which is concave if and only if $c\in [0,1/2]$.  

For \textbf{Example 2}, we allow for general functions $w_i(e)$ and assume that
$\wh S_2$ has the form
\begin{equation}
  \label{S}
  \wh S_2(\bfu,e)=  s(e) - \sum_{i=1}^I w_i(e) \LB\big( u_i/w_i(e)
\big) = \wh s(e) -\sum_{i=1}^I \Big(\LB(u_i)-u_i\log w_i(e)
\Big), 
\end{equation}
where $ \wh s(e)=s(e)+I-\sum_{i=1}^I w_i(e)$. Since
$\LB'(\nu)=\log \nu=0$ if and only if $\nu=1$, the reference densities
in the detailed balance condition \eqref{eq:DetBalance} are exactly
$\bfw(e)$. In addition to the dependence on the internal energy (i.e.\
on the temperature), they may vary with $x\in \Omega$. The
concavity can be checked by calculating
\[
- \binom\bfmu\eps \vdot \rmD^2 \wh S_2(\bfu,e) \binom\bfmu\eps = 
\sum_{i=1}^I u_i\Big(\frac{\mu_i}{u_i} - \eps
\frac{w'_i(e)}{w_i(e)}\Big)^2 + \eps^2 \, \Big( {-}\wh s''(e)- \sum_{i=1}^I
u_i\frac{w''_i(e)}{ w_i(e)}\Big). 
\]
Thus, we have strict concavity on the whole domain ${]0,\infty[}^I\ti
{]e_0,\infty[}$ if and only if $\wh s''(e)<0$ and $w''_i(e)\leq 0$ for all
$i$. 

Thus, the form \eqref{S} provides a quite general form to model
suitable temperature-dependent reaction-diffusion systems. A simple
choice  for $s(e)$ and  $\bfw(e)$ are therefore given by
\begin{equation}
  \label{s}
\begin{aligned}
& \wh s(e)= c\log e  \text{ for }c>0 \quad \textnormal{or} \quad \
\wh s(e)=c e^{\sigma} \text{ for }\sigma \in {]0,1[}, \\
&w_i(e)= u_{*,i} e^{b_i} \quad \text{for some } u_{*,i}>0 \text{ and }
b_i\in [0,1].
\end{aligned}
\end{equation}
In the case $s(\ie)= c\log e$ we find the simple relation 
$1/\theta =\pl_e \wh S(\bfu,e)= \big(c{+} \bfb {\vdot} \bfu)/e$,
where $\bfb=(b_i)_{i=1,...,I}$,
and we have the nice linear relation
\[
e = \ol E(\bfu,\theta)= \big(c{+} \bfb {\vdot} \bfu)\,\theta.
\]

\subsection{Comparison with energy transport models derived from the
  semiconductor Boltzmann equation} 
\label{ss:CompSCBoltz}

In the context of semiconductor modeling, energy transport equations
were derived in the diffusive scaling limit of the Boltzmann equation,
see, e.g.\ \cite{Juengel09, Juengel10}. The general (unipolar) form
obtained by this procedure is 
\begin{subequations}
\label{eq:n-e}
\begin{align}
\dot n &= \div \bfj_n, \label{SemET1}\\
\dot e &= \div \bfj_e - \bfj_n \cdot\grad_x V - W, \label{SemET2} 
\end{align}
with $n$ and $e$ the electron position and electron energy
densities, respectively.  The term $W$ describes averaged inelastic particle
scattering and $V$ is the electrostatic potential.  The particle and
energy current densities $\bfj_n$ and $\bfj_e$ are expressed as
functions of the chemical potential $\mu$ and electron temperature
$\theta$,
\begin{align}
    \bfj_n &= D_{00} \left( \grad_x \left(\frac{\mu}{\theta}\right) - \frac{\grad_x V}{\theta} \right) - D_{01} \grad_x \left(\frac{1}{\theta}\right),\\
    \bfj_e &= D_{10} \left( \grad_x \left(\frac{\mu}{\theta}\right) - \frac{\grad_x V}{\theta} \right) - D_{11} \grad_x \left(\frac{1}{\theta}\right),
\end{align}
\end{subequations}
where $D_{ij}=D_{ij}(\mu,\theta)$, $i$, $j\in \{0,1\}$, 
defines the symmetric and positive definite diffusion matrix.
For the sake of the comparison with our formulation \eqref{eq:RDS-u-e}
we set $W=0$.

Under reasonable simplifying assumptions
\cite{Juengel09, Juengel10} the extensive variables $(n,e)$ and the
intensive variables $(\mu, \theta)$ are related by the formulas
\begin{equation}
  \label{eq:Juengel}
    n = N_0 \theta^{3/2}\exp(\mu/\theta) \ \text{ and } \ e = \frac{3}{2} n\theta,
\end{equation}
where $N_0 \theta^{3/2} $ is the scaled density of states.
The diffusion matrix $(D_{ij})$ then becomes
\[
    (D_{ij}) = \mu_0 \Gamma(2{-}\beta) n\theta^{1/2-\beta}
      \begin{pmatrix}
       1 & (2{-} \beta)\theta \\ (2 {-} \beta)\theta & 
  (3{-}\beta)(2{-}\beta)\theta^2
      \end{pmatrix},
\]
where $\mu_0>0$ comes from the elastic scattering rate and $\Gamma$
denotes the Gamma function.  Typical choices for $\beta$ are
$\beta=1/2$ (nonpolar phonon scattering), $\beta=0$ (acoustic phonon
scattering) and $\beta=-1/2$ (diffusion approximation of the
hydrodynamic semiconductor model), see \cite{Juengel10}.

We now show that the above system \eqref{eq:n-e} is consistent
with our modeling scheme with
\[
\calS(n,e)=\int_\Omega \wh S(n(x),e(x))\dd s  \ \text{ and } \ 
\calE(n,e)= \int_\Omega e(x) - n(x) V(x) \dd x ,
\]
where using the relation $\mu=-\theta \pl_n\wh S(n,e)$ the choice 
\[
\wh S(n,e)=\frac32 n \log e - \frac52 \big(n\log n -n\big)+ n
\log\big(N_0 (2/3)^{3/2}\big) 
\]
is consistent with \eqref{eq:Juengel}.  Example 1 in
Section \ref{ss:Examples} shows that $\wh S$ is concave. We define the
dual entropy-production 
potential via 
\[
\Psi^* (n,e;\nu,\eps)= \frac12\int_\Omega \binom{\nabla \nu + \eps
  \nabla V}{\nabla \eps} \cdot \binom{D_{00} \quad D_{01}}{D_{10}\quad
  D_{11}} \binom{\nabla \nu + \eps
  \nabla V}{\nabla \eps} \dd x ,
\]
where we emphasize that the term $\eps\nabla V$ (in contrast to
$\nabla ( \eps V)$\,) is important to obtain the above model. Using
$\Psi^*(n,e;\rmD\calE(n,e))=0$ we conclude that the associated Onsager 
operator $\bbK(n,e)=\rmD^2\Psi(n,e)$ satisfies the energy conservation
$\bbK(n,e)\rmD\calE(n,e)=0$. Indeed, we have 
\[
 \bbK(n,e)\binom{\nu}\eps = \binom{-\div\big( D_{00}(\nabla \nu+\eps
   \nabla V)+ D_{01} \nabla \eps\big)}{
- \div\big( D_{10}(\nabla \nu+\eps \nabla V)+ D_{11}\nabla \eps\big) +
\nabla V\cdot\big(D_{00}(\nabla \nu+\eps \nabla V)+D_{01}\nabla \eps\big)}.
\]
Inserting $\nu=-\mu/\theta= \pl_n\wh S(n,e)$ and $\eps= \pl_e\wh
S(n,e)=1/\theta$, we see that system \eqref{eq:n-e} with $W\equiv 0$
is a gradient system for the entropy $\calS$ and the Onsager operator
$\bbK$.

\section{Maximum principle and evolution for explicit and special cases}
\label{s:Special}

\subsection{Maximum entropy principle} 
\label{ss:MaxEntropy}
For the general system we may consider all conservation laws using the projection
$\calC_{\bbS^\bot}$ as defined in \eqref{eq:calP--bbS-bot}. Hence, for
given values of 
\begin{equation}\label{cons.laws}
\calC_{\bbS^\bot}(\bfu)=\int_\Omega \Proj  \bfu(x)\dd x =: \bfC_0\in \bbS^\bot \quad \text{and} \quad 
\calE(\bfu,e)= \int_\Omega e(x) \dd x =: E_0
\end{equation}
we maximize the concave function
\[
(\bfu,e) \mapsto \widehat\calS(\bfu,e).
\]
Under reasonable weak additional assumptions, we obtain a unique
maximizer $(\bfu^*,e^*)$, called thermodynamical equilibrium.

If $\bfw(e)$ is independent of $x\in \Omega$, then
$(\bfu^*,e^*)$ is spatially constant. If $\bfw$ and hence $\wh
S(x,\bfu,e)$ depends on $x\in \Omega$, then $(\bfu^*,e^*)$ is a
nonconstant function on $\Omega$. Nevertheless, the temperature
$\theta^*$ is constant as 
\begin{eqnarray}\label{max.S}
\pl_{\bfu} \wh\calS(x,\bfu^*(x),e^*(x))={\bfSigma}_{\bfu} \Proj \quad \text{and} \quad
1/\theta^* = \pl_e \wh S(x,\bfu^*(x),e^*(x)) =\Sigma_{\ie},
\end{eqnarray}
where $\bfSigma_{\bfu}\in \bbS^\bot$ is the constant Lagrange
multiplier associated with $\calC_{\bbS^\bot}$ and $\Sigma_{\ie} \in
\R$ is the constant Lagrange multiplier associated with
$\calE(\bfu,e)= \int_\Omega e(x) \dd x = E_0$.

\subsection{The case of $x$-independent $\bfw(e)$}

The simplest reaction-diffusion system is obtained if we choose the mobility tensor $\wh
\bbM=-\kappa(\rmD^2\wh S)^{-1}$, where $\kappa>0$ is a scalar. Note that
$\wh \bbM$ has to be a  symmetric tensor (of order 4), so we are not
able to choose $\kappa$ to be a symmetric tensor not commuting with
$\rmD^2 \wh S$. Of course $\kappa$ may depend on $x$, but we assume it
to be constant for simplicity.
Inserting this mobility tensor into \eqref{eq:RDS-u-e} we obtain the system
\begin{equation} 
  \label{eq:RDS2}
\begin{aligned}
\dot \bfu &= \kappa \Delta \bfu + \wh\bbH(\bfu,e) \pl_{\bfu}\wh S(\bfu,e),\\
\dot e & = \kappa \Delta e. 
\end{aligned} 
\end{equation}
Thus, our model, which is thermodynamically consistent,
predicts that the internal energy diffuses independently of the
densities $u_i$, however, there may still be a strong dependence of
the reactions on the internal energy.  

We note that the reaction terms are given via $\wh\bbH$ as in
\eqref{eq:2OnsagerMatrix}, where $\bfu^*$ is replaced by $\bfw(e)$ and
$k^*_r$ can be chosen arbitrarily as a function of $(\bfu,e)$,
giving 
\begin{equation} 
  \label{eq:RDS3}
\wh\bbH(\bfu,e)\pl_{\bfu} \wh S(\bfu,e)= -\sum_{r=1}^R k^*_r(\bfu,e)\big(
\tfrac{\bfu^{\bfalpha^r}}{\bfw(e)^{\bfalpha^r}} - 
\tfrac{\bfu^{\bfbeta^r}}{\bfw(e)^{\bfbeta^r}}\big) \big(\bfalpha^r -
\bfbeta^r \big).
\end{equation}
Imposing sufficient decay to $k^*_r(\bfu,e)$, it is possible to make
these terms globally Lipschitz if necessary. 

The constant steady state is determined as the maximum of $\wh\calS$  subject to the conservation properties \eqref{cons.laws}.
With $\wh \calS$  given by \eqref{S}, the relations \eqref{max.S} become
\begin{eqnarray*}
&&\log \tfrac{u_i^*}{w_i(\ie^*)}=-( \mathbf{\Sigma_{u}} \mathbb{P} )_i=:\log \wt C_i\\
&&s'(\ie^*)+\sum_{i=1}^Iu_i^*\tfrac{w_i'(e^*)}{w_i(e^*)}=\Sigma_e
\end{eqnarray*}
Inserting the first relation into the second one we obtain
\begin{eqnarray*}
s'(\ie^*)+\tfrac{\mathbf{u}^*\cdot \mathbf{b}}{e^*}=s'(\ie^*)+\sum_{i=1}^I\wt C_i w_i'(\ie^*)=\Sigma_\ie\,.
\end{eqnarray*}
The conservation property of the internal energy determines $e^*$,
which then sets the Lagrange multiplier $\Sigma_\ie$.  The other
Lagrange multipliers $ \mathbf{\Sigma_{u}}$, and therefore $\wt C_i$,
are determined through the conservation law for
$\mathbb{P}\mathbf{u}$. The relative entropy functional
\begin{eqnarray*} 
  \calH (\mathbf{u},\ie)=-\wh\calS(\mathbf{u},\ie) +
  \wh\calS(\mathbf{u}^*,\ie^*)+\mathbf{\Sigma_{u}}\int_{\Omega} 
 \mathbb{P}(\mathbf{u}-\mathbf{u}^*)\dd x 
  + \Sigma_e \int_{\Omega}(e-e^*)\dd x
\end{eqnarray*}
decays until the steady state $(\mathbf{u}^*,\ie^*)$ is reached. Using
the relations for the Lagrange multipliers we can reformulate
$\calH $ as follows
\begin{eqnarray*}
  \calH (\mathbf{u},\ie)&=
 &\sum_{i=1}^I\int_\Omega \widetilde C_i w_i(\ie)
 \LB\big(\tfrac{u_i}{\widetilde C_i w_i(\ie)}\big) \dd x -
 \sum_{i=1}^I\int_\Omega \widetilde C_i
 \big(w_i(\ie)-w_i(\ie^*)-w_i'(\ie^*)(\ie-\ie^*)\big) \dd x  \\
  &&\qquad - \sum_{i=1}^I\int_\Omega
  \big(s(\ie)-s(\ie^*)-s'(\ie^*)(\ie-\ie^*)\big)\dd x \,.
\end{eqnarray*}
This relation holds generally for the entropy $\wh {\mathcal S}$ with
steady states being the maximizers under the conservation laws
\eqref{cons.laws}. We see that due to the concavity of $s$ and $w_i$
with respect to $\ie$ also the last two terms give nonnegative
contributions to the relative entropy.

With the particular choice 
\[
w_i(\ie)=C_i \ie^{b_i}\,, \qquad s(\ie)=c \ie^\sigma \,,
\]
we can simplify the relative entropy to 
\begin{align}
\calH (\mathbf{u},\ie)&=\sum_{i=1}^I\int_\Omega \widetilde C_i
w_i(e) \LB\big(\frac{u_i}{\widetilde C_i w_i}\big)
\dd x + \sum_{i=1}^I\int_\Omega \frac{C_i}{(\ie^*)^{1-b_i}}\big(b_i \ie
-\ie^{b_i}(\ie^*)^{b_i-1}+ (1{-}b_i)e^*\big)\dd x  \nonumber\\ 
\label{F.ind.gen}&\qquad +  \int_\Omega\frac{c}{(\ie^*)^{1-\sigma}}\big(\sigma  \ie -\ie^{\sigma}(\ie^*)^{\sigma-1}+
(1{-}\sigma)e^*\big)\dd x \,, 
\end{align}
where the nonnegativity of the last two integrands can be checked
directly by applying Young's inequality.

\subsection{The case of $x$-dependent $\bfw(x,e)$}
\label{sec.x.dep}

We let again $\wh S$ be of the form \eqref{S}, where we now assume
\begin{equation}\label{w.V.gen}
w_i(x,e) = \exp(-V_i(x)) e^{b_i},
\end{equation}
for suitable potentials $V_i:\Omega \to \R$ with $\Omega\subset
\mathbb{R}^d$.  In the case of potentials being involved, the
stationary state is not constant anymore.  Therefore an appropriate
choice of the energy density $s(x,\ie)$ is more delicate.  In order
for $\wh S$ to possess a maximum satisfying the required conservation
properties, we need the equations \eqref{max.S} to hold for the steady
state. As above the $x-$dependent stationary state $(n^*,p^*,e^*)$ has
to satisfy
\begin{eqnarray}
  \label{max.S.x.u}
  &&\log \frac{u_i^*}{w_i(x,\ie^*)}=-(\mathbf{\Sigma_{u}}\mathbb{P}
  )_i=:\log \wt C_i\,,\\ 
\label{max.S.x}&& s'(x,\ie^*)+\sum_{i=1}^I\wt C_i
w_i'(x,\ie^*)=\Sigma_\ie ,
\end{eqnarray}
where here and in the following $s'(x,\ie)=\pl_\ie s(x,e)$ and
$w_i'(x,e)=\pl_e w_i(x,\ie)$.
We give two generalizations for $s(x,e)$ of the power law form in
\eqref{s} that allow us to deduce the existence of a maximizer of the
entropy $\widehat {\mathcal S}$. The first one is given by
\begin{equation}
  \label{s.x.1} 
s(x,\ie)=c \ie^\sigma\gamma^{1-\sigma}\,,\qquad \sigma \in (0,1)\,,
\end{equation}
where $\gamma(x)\geq 0$ is integrable and w.l.o.g. assumed to be
normalized such that
\begin{equation}\label{gamma.int}
\int_\Omega \gamma (x) \dd x =\int_\Omega \ie \, \dd x = E_0\,.
\end{equation}
Then condition \eqref{max.S.x} becomes
\[ 
c \sigma \big(\tfrac{\ie^*}{\gamma}\big)^{1-\sigma}+\sum_{i=1}^I\wt
C_i b_i (\ie^*)^{b_i-1}\exp(-V_i)=\Sigma_{\ie}\,.
\]
Both terms on the left hand side are nonnegative. In general the terms
cannot balance each other such that the sum of both terms adds up to a
constant. Hence each of them will have to be constant individually,
showing that
\[\ie^*(x)=\gamma(x)\]
 and restricting the choices of the potentials $V_i$ and  exponents $b_i$.
This leads to the relative entropy functional 
\begin{align}
\calH (\mathbf{u},\ie)&=\sum_{i=1}^I\int_\Omega \widetilde C_i
w_i(x,e)  \LB\big(\tfrac{u_i}{\widetilde C_i w_i}\big) \dd x \nonumber\\
&\quad + \sum_{i=1}^I\int_\Omega  \exp(-V_i(x))(\ie^*)^{b_i-1}\big(b_i \ie -\ie^{b_i}(\ie^*)^{b_i-1}+ (1-b_i)e^*\big)\dd x  \nonumber\\
\label{F.ind}&\quad + c \int_\Omega \big(\sigma
\ie -\ie^{\sigma}(\ie^*)^{\sigma-1}+ (1-\sigma)e^*\big)\dd x \,. 
\end{align}

Another generalization of the power law in \eqref{s} motivated by
\eqref{max.S.x} is given by 
\begin{eqnarray*}
s(x,\ie)=\sum_{i=1}^I c_i w_i(x,e)\,, \qquad \textnormal{for} \ \ c_i\geq 0\,.
\end{eqnarray*}
The steady state relation \eqref{max.S.x} then reduces to 
\begin{eqnarray}\label{ass.gen.2}
\sum_{i=1}^I (c_i+\wt C_i)  w_i'(e^*,x)=\sum_{i=1}^I (c_i+\wt C_i) b_i \ie^{b_i-1} \exp(-V_i)=\Sigma_e\,,
\end{eqnarray}
which relates the stationary state $\ie^*$ to the potentials $V_i$
and, again, induces a restriction on the choices of $b_i$ and
$V_i$. Relation \eqref{ass.gen.2} allows us to rewrite the relative
entropy as follows
\begin{align*}
\calH(\mathbf{u},e) &= - \wh{\cal S}(\mathbf{u},e)+\wh {\cal S}(\mathbf{u}^*,e^*) +{\bfSigma_u} \int_{\mathbb{R}^d} \mathbb{P} (\mathbf{u}-\mathbf{u}^*) \dd x +\Sigma_e\int_{\mathbb{R}^d} (e-e^*) \dd x\\
&=\sum_{i=1}^I \int_\Omega \wt C_i w_i(e) \LB\big(\frac{ u_i}{\wt C_i
  w_i(e)}\big) \dd x\\
&\quad + \sum_{i=1}^I(c_i {+}\wt C_i)\int_{\mathbb{R}^d}
(e^*)^{b_i-1}\exp(-V_i)\left(b_i e - e^{b_i} (e^*)^{1-b_i} + 
 (1 {-} b_i) e^*\right) \dd x   
\end{align*}

Assuming now that $s(x,\ie)$ is chosen appropriately such that the entropy $\wh {\cal S}$ can be maximized under the conservation laws,
we derive the corresponding evolutionary system according to \eqref{eq:RDS-u-e}. We therefore first differentiate
\[
\nabla \binom{\pl_{\bfu} \wh S(x,\bfu,e)}{\pl_e \wh S(x,\bfu,e)} = \rmD^2
\wh S(x,\bfu,e) \nabla \binom{\bfu}e +\nabla_{(x)} \binom{ -\bfV}{s'(x,\ie)},
\]
where $D^2 \wh {\cal S}$ denotes as above the Hessian of $\wh {\cal S}$ with respect to $(\mathbf{u},\ie)$ and here and in the following we denote 
\[
\nabla_{(x)} =\nabla \bigg|_{\ie=\textnormal{const.}}\,.
\]
The choice $\wh\bbM= - \kappa ( \rmD^2 \wh S)^{-1}$ for the mobility tensor
leads to the system
\begin{equation} 
  \label{eq:RDS4}
\begin{aligned}
\dot \bfu &= \kappa \big(\Delta \bfu +\div \bfD^\bfu  \big) + 
\wh\bbH(\bfu,e) \pl_{\bfu}\wh S(\bfu,e),\\
\dot e & = \kappa \big(\Delta e + \div \bfd^e\big) ,
\end{aligned} 
\end{equation}
where the drift fluxes $\bfD^\bfu$ and $\bfd^e$ have the form 
\begin{eqnarray*}
\bfD^\bfu_{ik} &=& u_i\big(\pl_{x_k} V_i+ \sum_{j=1}^J \tfrac{u_j b_j
  b_i}{M(\bfu,\ie)} \pl_{x_k} V_j  - b_i \tfrac{e}{M(\mathbf{u},\ie)}\pl_{(x_k)} s'(x,\ie)\big)\\\
\bfd^\ie_k &=& \ie\sum_{j=1}^I \tfrac{b_j u_j}{M(\bfu,\ie)} \pl_{x_k} V_j - \tfrac{\ie^2}{M(\mathbf{u},\ie)}\pl_{(x_k)} s'(x,\ie). 
\end{eqnarray*}
Here and in the following $M(\bfu, \ie )=-\ie^2  s''(\ie)+ \sum_{i=1}^I u_i(b_i{-}b_i^2)$ and, as above,
$\pl_{(x_k)} s'(x,\ie)$ denotes the partial derivative of $s'$ with respect to $x_k$ while $\ie$ is kept constant. 
To see how the drift fluxes arise, we note that
\[
- \rmD^2 \wh S= \bma{cc} \delta(\bfu)^{-1} & - \tfrac1\ie \bfb\\ -
\tfrac1e \bfb^\top & \tfrac{-\ie^2s''(\ie)+ \bfb\cdot \bfu}{e^2}\ema 
\]
with  $\delta(\bfu)=\mathrm{diag}(\bfu)\in \R^{I\ti I}$. The inverse can then be calculated as
\[
- (\rmD^2 \wh S)^{-1} = 
\bma{cc} \delta(\bfu)+\frac1{M(\bfu,\ie)}
     (\delta(\bfu)\bfb){\otimes}( \delta(\bfu) \bfb)
  & \tfrac{e}{M(\bfu,\ie)} \delta(\bfu) \bfb
\\ 
\tfrac{e}{M(\bfu,\ie)} ( \delta(\bfu) \bfb)^\top &
  \tfrac{\ie^2}{ M(\bfu,\ie)}\ema \,.
\]


\subsection{An  $x$-dependent bipolar model with semiconductor-type  reactions}
\label{sec.x.dep.0}

We now write down the simplest system with two species, electrons and
holes, with the semiconductor-type reaction $X_n+X_p \REVER
\emptyset$.  The variables $n,p$ and $e$ stand for the density of
electrons, holes, and the internal energy. We set
\[
w_n(x,e)=\exp(-V_n(x))\sqrt{e}\,,\qquad w_p(x,\ie)=\exp(-V_p(x))
  \sqrt{\ie}\,,
\] 
and obtain, using $b\LB(a/b)=\LB(a) -a\ln b +b-1$,  
\begin{align} \nonumber
  \wh S(x,n,p,\ie)&= s(x,\ie) -
  w_n(x,\ie)\LB\big(\tfrac{n}{w_n(x,\ie)}\big) 
  -w_p(x,\ie)\LB\big(\tfrac{p}{w_p(x,\ie)}\big)  \\
\label{eq:whS}
  &= \wh s(x,\ie) +\tfrac12(n{+}p)\log \ie - \LB(n) -
  \LB(p) - V_n(x)    n - V_p(x) p, 
\end{align}
where $\wh s(x,e)=s(x,e)-w_n(x,e)-w_p(x,e)+2$.
The conserved quantities we denote by
\begin{equation}\label{cons.laws.0}
  \calC(n,p,e)=\int_\Omega ( n{-}p)\, \dd x = :\bfC_0\quad \text{and} \quad 
  \calE(n,p,e)=\int_\Omega \ie \,\dd x = :E_0 .
\end{equation}
For given $\bfC_0$, $E_0>0$ and an appropriate choice of $s(x,\ie)$ as
discussed above, there is a unique maximizer of $\wh
\calS(n,p,\ie)=\int_\Omega \wh S(x,n,p,\ie)\dd x$ subject to the
constraint $\calC(n,p,\ie)=\bfC_0$ and $\calE(n,p,\ie)=E_0$.

We choose the mobility tensor as $\wh \bbM=-\kappa (\rmD^2 \wh
S)^{-1}$. To close the dynamics it remains to set up the reaction
terms.  The typical form of the semiconductor reactions is given by
the Read-Shockley-Hall term $k(n_I^2 - np)$ with $k=k(x,n,p,\theta)$
the positive reaction rate and $n_I = n_I(\theta)$ the intrinsic
carrier density, see \cite{MarRinSchm90}.  The dependence of the
intrinsic density on the temperature is modeled as $n_I(\theta) = c_1
\theta^{3/2} \exp(-c_2/\theta)$ for some positive constants $c_1$,
$c_2$. Observe that $n_I$ is an increasing function of $\theta$.  For
the sake of simplicity of our forthcoming analysis, we set $n_I$ to
depend linearly on $e$.  This leads to the form
\[
k(\ie-\rho(x)np), \qquad\textnormal{with} \quad  \rho(x)=\exp(V_n(x)+V_p(x))\,.
\]
In order to see that this reactive term corresponds to the symmetric
form $\widehat \bbH\bfmu$ as in \eqref{eq:RDS3}, where we recall the
notation $\bfmu=\partial_{(n,p)} {\widehat{\mathcal{S}}}$, we first
rewrite
\begin{eqnarray*}
k(\rho(x) n p - \ie)= r_0(n,p,\ie)\big(\ln \tfrac{ n }{w_n(\ie)}+\ln \tfrac{p}{w_p(\ie)}\big),\quad 
\textnormal{with}
\quad r_0=k \ie\,\big(\tfrac{ n p}{w_nw_p} - 1\big)/\ln \tfrac{ n p}{w_nw_p}\geq 0\,.
\end{eqnarray*}
Thus, we have
\[
   k(\ie - \rho(x) n p )\binom{1}{1}=\wh\bbH \bfmu, \quad \textnormal{where} \quad  \wh\bbH=r_0 \binom{1\ 1}{1\ 1}
\]
and we obtain the evolution equations 
\begin{align}
\nonumber
\dot n& =\kappa \Big(\Delta n + \div\big(n\nabla V_n +
\tfrac{n}{N}\big(n\nabla V_n {+}
 p \nabla V_p - 2 \ie \nabla_{(x)} s'(x,\ie) \big) \Big) \!
      + k(x,n,p,e)\big( e {-}\rho(x) np\big) , \\
\nonumber
\dot p& =\kappa \Big(\Delta p + \div\big(p\nabla V_p + \tfrac{p}{N}\big(n\nabla V_n+
 p \nabla V_p - 2 \ie \nabla_{(x)} s'(x,\ie) \big)\Big) 
          + k(x,n,p,e)\big( e {-}\rho(x) np\big),\\
\nonumber
\dot e& =\kappa \Big(\Delta e + \div\big(\tfrac{2e}{N} \big(n\nabla V_n+
 p \nabla V_p - 2 \ie \nabla_{(x)} s'(x,\ie) \big)\Big),\\
0&= \nabla p\cdot \nu= \nabla n\cdot \nu=\nabla e\cdot \nu \quad
\text{ on }\pl \Omega,  \label{eq:4.npe}
\end{align}
where $N(n,p,e)=4M(n,p,e)=-4 e^2 s''(x,e)+n+p$ and $\rho(x) =
\exp(V_n(x){+}V_p(x))$.  The reaction terms arise as in
\eqref{eq:RDS3} using $\bfalpha_r=(1,1)^T,
\bfbeta_r=(0,0)^T$.  The reaction coefficient $k(x,n,p,e)>0$ can
be chosen arbitrarily, for instance the Read-Shockley-Hall
generation-recombination model gives 
\ $k=k_0/(1{+}c_n n + c_p p)$ for positive constants $k_0,c_n$, and
$c_p$.


\section{Global existence of solutions and convergence to equilibrium for particular systems}\label{s:Convergence}

In this section we derive entropy entropy-production inequalities to
prove convergence towards the stationary state in two particular
systems. Our approach is inspired by the work \cite{DiF-Fel-Mar} and
uses logarithmic Sobolev inequalities to bound the entropy in terms of
the entropy production.  With the known functional inequalities this
is, even in the $x$-independent case, not possible for an entropy term
of the form $s(\ie)=c\ln e$.  However, the alternative choice
$s(\ie)=c \ie^{\sigma}$ with $c\geq0$ satisfies all required
thermodynamical properties and allows us to establish exponential
decay of the corresponding relative entropy. In the following we
therefore focus on the cases
\begin{equation}\label{equ.s}
s(\ie)=c\sqrt{e}\,,\quad \textnormal{resp.}\quad
s(x,\ie)=c\sqrt{\ie}\sqrt{\gamma}\,, 
\end{equation} 
where $\gamma(x)\geq 0$ verifying \eqref{gamma.int} corresponds to
$e^*(x)$ in the $x-$dependent case.

\subsection{Global existence and long time behavior in the $x$-independent case} \label{ss:x-indep}

\subsubsection{Steady states, relative entropy and evolution equations}
We consider a bipolar model with the semiconductor-type reaction $X_n+X_p \REVER \emptyset$. 
We put $\bfu=(n,p)$, so that the variables are $(n,p,e)$ for the density of electrons, holes,
and internal energy, respectively. For $C_p$, $C_n>0$ we set
\begin{eqnarray*}
w_n(e)=C_n\sqrt{e},\qquad w_p(e)=C_p\sqrt{e}\,.
\end{eqnarray*}
The derivation of the evolutionary equations from the entropy
functional leads to the following semilinear reaction-diffusion system
\begin{subequations}
\label{eq:rdM}
\(
    \partial_t n &=& \kappa \laplace n + k(\ie-np),  \label{rdM1} \\
    \partial_t p &=& \kappa \laplace p + k(\ie-np),  \label{rdM2} \\
    \partial_t \ie &=& \kappa \laplace\ie   \label{rdM3}
\)
\end{subequations}
for the constant diffusion coefficient $\kappa>0$.
The system is posed on the $d$-dimensional torus $\calT^d$ (i.e., with periodic boundary conditions),
rescaled such that $|\calT^d|=1$.

We shall now investigate the steady states and define the relative entropy. 
We therefore maximize the entropy 
\[
\wh S(n,p,\ie)=c\sqrt{\ie}- \LB(n) -\LB(p) + n 
\ln w_n(\ie) + p \ln w_p(\ie)
\]
on $\calT^d$ under the conservation laws
\begin{eqnarray}\label{normalizations}
 \int_{\calT^d} (n-p) \dd x=  \int_{\calT^d} (n_0-p_0) \dd x = \bfC_0\,,\qquad \int_{\calT^d} \ie \dd x = \int_{\calT^d} \ie_0 \dd x =E_0\,.
\end{eqnarray}
Introducing the Lagrange multipliers $\Sigma_0$ and  $\Sigma_{\ie}$, the steady state $(n^*,p^*,\ie^*)$ is
determined via
\begin{eqnarray*}
-\partial_n\wh S +\Sigma_0 =0\,,\qquad -\partial_p\wh S -\Sigma_0=0\,,\qquad -\partial_\ie\wh S +\Sigma_{\ie}=0\,,
\end{eqnarray*}
implying the following relations
 \begin{eqnarray*}
&&n^*=w_n(\ie^*)\exp(-\Sigma_0)=C_n\sqrt{\ie^*}\exp(-\Sigma_0)=\widetilde C_n \, C_n \sqrt{\ie^*}\,,\\
&&p^*=w_p(\ie^*)\exp(\Sigma_0)=C_p\sqrt{\ie^*}\exp(\Sigma_0)=\widetilde C_p \, C_p \sqrt{\ie^*}\,,\\
&&\tfrac{n^*+p^*}{2\ie^*}+\tfrac{c}{2\sqrt{\ie^*}}=\Sigma_{\ie}\,.
\end{eqnarray*}
The detailed balance condition requires
\begin{equation}\label{det.bal.np}
n^*p^*=e^*, \qquad \textnormal{implying} \quad C_nC_p=1\,.
\end{equation}
Moreover, from the conservation property of $n-p$ we have
\[
   \tfrac{1}{2\sqrt{\ie^*}}(n^*{-}p^*)=\tfrac{1}{2}\big(\exp(\log C_n
   {-} \Sigma_0)- \exp(-(\log C_n {-}\Sigma_0))\big)
   =\sinh({-}\log C_n + \Sigma_0) =\tfrac{\bfC_0}{2\sqrt{\ie^*}}.
\]
We note that the pair of constants $C_n\widetilde C_n=C_n\exp(-\Sigma_0), C_p\widetilde C_p = C_p \exp(\Sigma_0)$ 
satisfies the detailed balance condition \eqref{det.bal.np}. Therefore we shall in the following assume w.l.o.g.  $\widetilde C_n=\widetilde C_p=1$, which amounts to setting $\Sigma_0=0$. This can be understood in the sense that the constants $C_n$ and $C_p=C_n^{-1}$ are already the right weights in $w_n$ and $w_p$ for the stationary states $n^*$ and $p^*$, i.e.
\[
   n^*=C_n \sqrt{\ie^*}, \qquad p^*=C_p \sqrt{\ie^*}.
\]
The constant steady state, and therefore also the constants $C_n$ and $C_p$, are uniquely determined through the conservation laws and the detailed balance condition, since 
\[
\ie^*=E_0\,,\qquad C_n-C_p=C_n-\tfrac{1}{C_n} = \tfrac{\bfC_0}{\sqrt{\ie^*}}\,.
\]
The Lagrange multiplier $\Sigma_{\ie}$ is given by
\begin{equation}\label{Sigma}
\Sigma_{\ie} =\frac{C_n+C_p+c}{2\sqrt{\ie^*} }\,. 
\end{equation}
Recalling $\Sigma_0=0$, the convex relative entropy reads as follows
\begin{eqnarray}
\calH (n,p,\ie)
\label{F1}&=&-\wh S(n,p,\ie) + \wh S(n^*,p^*,\ie^*) +
\Sigma_{\ie}\int_{\calT^d}(\ie-\ie^*)\dd x\,,
\end{eqnarray}
which has the crucial property $\calH(n,p,\ie)\geq
\calH(n^*,p^*,e^*) =0$.
Direct computation or using $\widetilde C_i =1, \sigma=b_i=1/2$ in
\eqref{F.ind.gen} shows that it can also be formulated as
\begin{eqnarray}
 \calH (n,p,\ie)=\int_{\calT^d}
w_n(\ie)\LB\big(\tfrac{n}{w_n(\ie)}\big) 
+ w_p(\ie)\LB\big(\tfrac{p}{w_p(\ie)}\big) 
+\tfrac{c+C_n+C_p}{2\sqrt{\ie^*}}
(\sqrt{\ie}-\sqrt{\ie^*})^2\dd x\,. \label{ent.F}
\end{eqnarray}
Note that one can equivalently derive the evolution equations
\eqref{eq:rdM} starting from $- \calH (n,p,e)$ instead of
$\wh \calS(n,p,\ie)$.

We will prove exponential convergence of solutions to
\eqref{eq:rdM} based on the dissipation relation
\begin{equation}
   \label{diss.rel}
\tot{}{t} \calH(n,p,\ie) = - \calP(n,p,\ie),\end{equation}
where the entropy production potential \eqref{ent.prod} reduces to
\[
    \calP(n,p,\ie) := \int_{\calT^d} \kappa \Big( n \left|\grad\ln
      \tfrac{n}{\sqrt\ie}\right|^2 + p \left|\grad\ln
      \tfrac{p}{\sqrt\ie}\right|^2  + \tfrac{N}{4}  \left|\tfrac{\grad
        \ie}{\ie}\right|^2 \Big) +  k(np{-}\ie) \ln\tfrac{np}{\ie} \d x\,, 
\]
with $N =N(n,p,e)= n+p+c\sqrt{\ie}$.

\subsubsection{Global existence of solutions}
We prove global well-posedness of the system \eqref{eq:rdM} on the
torus $\calT^d$ subject to the nonnegative initial datum
$(n_0,p_0,\ie_0)\in L^\infty(\calT^d)^3$ and the conservation laws
\eqref{normalizations} with $E_0=\ie^*$.  For the forthcoming analysis
we make the assumption that the initial energy is bounded from above
and away from zero, i.e. we assume there exist constants
$0<\underline\ie <\bar\ie$ such that
\begin{equation}\label{ass.e0}
    0 < \underline\ie \leq \ie_0(x) \leq \bar\ie \qquad\mbox{for all } x\in \calT^d.
\end{equation}
Then the maximum principle for the solutions $\ie$ of \eqref{rdM3} implies
\( \label{e_bbelow}
    0 < \underline\ie \leq \ie(t,x) \leq \bar\ie \qquad\mbox{for all } t\geq 0\mbox{ and } x\in \calT^d.
\)
We moreover let the reaction coefficient be bounded from above and away from zero by
\[
   0<\underbar k \leq k \leq \bar k\,.
\]

Solutions of the system \eqref{eq:rdM} preserve nonnegativity since the nonlinearities
on the right-hand side satisfy the \emph{quasi-positivity} condition, see e.g.\ Lemma 1.1 of \cite{Pierre}.
The maximum principle for 
\[
   \partial_t (n+p) - \kappa \laplace (n+p) = 2k(\ie-np) \leq 2\bar k\bar e
\]
gives
\[
   \sup_{t\in[0,T]} \Norm{n(t)+p(t)}_{L^\infty(\calT^d)} \leq 2\bar k \bar e T \Norm{n_0+p_0}_{L^\infty(\calT^d)},
\]
and by nonnegativity,
\[
   \sup_{t\in[0,T]} \max\{\Norm{n(t)}_{L^\infty(\calT^d)}, \Norm{p(t)}_{L^\infty(\calT^d)}\} \leq 2 \bar k \bar e T \Norm{n_0+p_0}_{L^\infty(\calT^d)}.
\]
This immediately implies the existence of global classical solutions, see \cite{Pierre}.
Moreover, the following slightly refined analysis shows that the $L^1$-norms of $n$ and $p$ are  uniformly bounded.

\begin{lemma} \label{lemma:L1bdd-x-indep} Let $(n,p,\ie)$ be a
  solution of the system \eqref{eq:rdM} subject to the nonnegative
  initial data $(n_0,p_0,\ie_0)$ satisfying assumption \eqref{ass.e0},
  the normalizations \eqref{normalizations} and
  $\calH(n_0,p_0,\ie_0)<\infty$.  Then 
\( \label{bound_L1_np}
  \sup_{t\geq 0} \big(\Norm{n}_{L^1(\calT^d)} +
  \Norm{p}_{L^1(\calT^d)} \big) < \infty.  \)
\end{lemma}
\begin{proof}
Defining $\xi:=n-p$, equations \eqref{eq:rdM} give
\(   \label{heat-xi}
    \partial_t \xi = \kappa \laplace \xi.
\)
Therefore, we may write \eqref{rdM2} as
$
   \partial_t p = \kappa \laplace p + k(\ie-p(p+\xi)),
$
and an integration over $\calT^d$ gives
\[
    \frac{\d}{\d t} \int_{\calT^d} p\, \d x = k\Big( \ie^* - \int_{\calT^d} p(p+\xi)\d x \Big),
\]
where we used the mass conservation of $\ie$.
The Cauchy-Schwarz inequality
$\int_{\calT^d} p\, \d x \leq \big(\int_{\calT^d} p^2 \, \dd x\big)^{1/2}$
and the global boundedness $|\xi(t,x)|\leq C$ (note that $\xi$ solves the heat equation \eqref{heat-xi}) gives
\[
    \frac{\d}{\d t} \int_{\calT^d} p \, \d x \leq k\left[ \ie^* - \Big(\int_{\calT^d} p \,\d x\Big)^2 + C\int_{\calT^d} p \, \d x \right],
\]
which implies
$    \sup_{t\geq 0} \int_{\calT^d} p \, \d x < \infty$.
Repeating the same steps for $n$, we obtain \eqref{bound_L1_np}.
\end{proof}

\subsubsection{Convergence to equilibrium}\label{ss:CtE:x-indep}

In order to prove the convergence of solutions to the stationary state
we have to show the decay of the relative entropy.  We shall
therefore split the dissipation term into four nonnegative parts,
where the ones resulting from the diffusion terms in the dynamics are
related to $-\calH$ by using the logarithmic Sobolev type of
inequalities \eqref{ineq:logSobolev} and the Sobolev imbedding
theorems \eqref{ineq:Sobolev}, where care has to be taken since the
norms $\Norm{n}_{L^1}$, $\Norm{p}_{L^1}$ are not conserved (in
contrast to $\Norm{\ie}_{L^1}$), see also
\cite{MiHaMa15UDER,DiF-Fel-Mar}. In particular, the difficulty of
treating the new mixed dissipation terms arising here involving the
heat component is overcome by applying the log-Sobolev inequality
with respect to the measure $\ie \d x$. This clearly
requires $\ie$ to be bounded uniformly from above and below by a
positive constant.  In order to control the resulting remainder of the
reactive term we proceed in a similar fashion to \cite{DiF-Fel-Mar},
where the case of a semiconductor reaction-diffusion system with a
confining potential but without a heat component was investigated.

\begin{proposition}[Entropy entropy-production estimate I]\label{prop:EEP}
  For all nonnegative\linebreak[4] $(n,p,\ie)$, for which $\calH(n,p,\ie)$,
  $\calP(n,p,\ie) < \infty$ and $\ie$ satisfies \eqref{e_bbelow}, there
  exists a $K$,
 \[  
K=K(\kappa,\Norm{n}_{L^1},\Norm{p}_{L^1},n^*,p^*,\underline{\ie},\bar{e},e^{*})>0
\]
such that the following estimate holds
\(   
\label{entropy-entropy_dis}
    \calH(n,p,\ie) \leq K \calP(n,p,\ie)\,.
\)
The explicit dependence of $K$ on its arguments is given in \eqref{K}.
\end{proposition}

Note that due to Lemma \ref{lemma:L1bdd-x-indep}, $K$ is uniformly
bounded along solutions, i.e.\    
\begin{equation}
\label{K.max}
\sup_{t>0} K(\kappa,\Norm{n}_{L^1}(t),\Norm{p}_{L^1}(t),n^*,p^*,\underline{\ie},\bar{e},e^{*}) =: \widehat K <\infty\,.
\end{equation}

\noindent\begin{proof}
Let us denote
\[
    \bar n := \int_{\calT^d} n \d x, \qquad  \bar p := \int_{\calT^d} p \d x.
\]
Using the identity
\[
    \int_{\calT^d} n\ln \big(\frac{n\sqrt{\ie^*}}{n^*\sqrt\ie}\big) \d x =
       \frac12 \int_{\calT^d} n\ln\frac{n}{\bar n} \d x + \frac12 \int_{\calT^d} n\ln\frac{n\ie^*}{\bar n\ie} \d x + \bar n\ln\frac{\bar n}{n^*}
\]
and its analog for $p$, we reformulate the relative entropy \eqref{ent.F} as
\begin{align}
   \calH(n,p,\ie) &= \frac12 \int_{\calT^d} n\ln\frac{n}{\bar n} \d x + \frac12 \int_{\calT^d} n\ln\frac{n\ie^*}{\bar n \ie} \d x
      + n^* \LB\big( \frac{\bar n}{n^*}\big)\nonumber \\
    &\quad + \frac12 \int_{\calT^d} p\ln\frac{p}{\bar p} \d x + \frac12 \int_{\calT^d} p\ln\frac{p\ie^*}{\bar p \ie} \d x
      + p^* \LB\big(\frac{\bar p}{p^*}\big) \nonumber\\
 \label{F.ind.0}   
 &\quad +\frac{c}{2\sqrt{e^*}} \int_{\calT^d} (\sqrt\ie - \sqrt{\ie^* })^2\d x.
\end{align}
Moreover, we split the entropy production into
\[\calP(n,p,\ie)=\kappa(\calP_n+\calP_p+\calP_\ie)+\calP_R\,,\] 
with
\[
    \calP_n = 2 \int_{\calT^d} \big|\grad\sqrt{n}\big|^2\d x + 2
    \int_{\calT^d}  \Big|\grad\sqrt{\tfrac{n\ie^*}{\ie}}\Big|^2
    \frac{\ie }{\ie^*}\d x,&&
\calP_\ie = 4c \int_{\calT^d} \big| \grad\sqrt[4]{\ie} \big|^2 \d x,\\
    \calP_p = 2 \int_{\calT^d} \big|\grad\sqrt{p}\big|^2 \d x + 2 \int_{\calT^d} \Big|\grad\sqrt{\tfrac{p\ie^*}{\ie}}\Big|^2 \frac{\ie }{\ie^*}\d x,
    &&
    \calP_R = - \int_{\calT^d} k(\ie - np) \ln\frac{np}{\ie} \d x.
\]
We apply the log-Sobolev inequality \eqref{ineq:logSobolev} to estimate the first two terms of $\calH(n,p,\ie)$ as
\[
    \frac12 \int_{\calT^d} n\ln\frac{n}{\bar n} \d x \leq \frac{C_{LS}}{2} \int_{\calT^d} \big|\grad\sqrt{n}\big|^2 \d x,
\]
and, using the fact that $\tfrac{\ie\d x}{\ie^*}$ is a probability measure on $\calT^d$, the generalized log-Sobolev inequality (see e.g.\ \cite{AMTU}) to obtain
\[
   \frac12 \int_{\calT^d} n\ln\frac{n\ie^*}{\bar n \ie} \d x 
   \leq \frac{C_{LS}({\ie}/{\ie^*})}{2} \int_{\calT^d} \Big|\grad\sqrt{\tfrac{n\ie^*}{\ie}}\Big|^2\frac{\ie }{\ie^*}\d x.
\]
Here $C_{LS}({e}/{e^*})$ is the log-Sobolev constant for the probability measure $\tfrac{\ie\d x}{\ie^*}$,
which depends on $\bar \ie$ and $\underline\ie$ and approaches the classical log-Sobolev constant $C_{LS}$
as $\ie$ converges to the stationary state.
Consequently, we have the bound
\[
  \frac12 \int_{\calT^d} n\ln\frac{n}{\bar n} \d x + \frac12 \int_{\calT^d} n\ln\frac{n\ie^*}{\bar n \ie} \d x \leq C_{LS}({\ie}/{\ie^*}) \calP_n,
\]
and the same estimate for the $p$-terms.

We now turn to the entropy term for the energy and first note that by the mass conservation law for $\ie$ we have
\begin{equation}\label{bound.F.n}
\frac{c}{2\sqrt{\ie^*}}\int_{\calT^d}\big(\sqrt{\ie}-\sqrt{\ie^{*}}\big)^2 \dd x = \frac{c}{\sqrt{\ie^*}}\int_{\calT^d}\big(\ie^*-\sqrt{\ie}\sqrt{\ie^{*}}\big) \dd x = c \int_{\calT^d}\big(\sqrt{\ie^*}-\sqrt{\ie}\big) \dd x\,.
\end{equation}
Moreover, the Jensen inequality gives
\[
   \int_{\calT^d} \sqrt{\ie^*} \d x \leq \Big( \int_{\calT^d} \ie^* \d x \Big)^{1/2} =\Big( \int_{\calT^d} \ie \d x \Big)^{1/2}
   = \Norm{\sqrt[4]{\ie}}_{L^4(\calT^d)}^2,
\]
and, subsequently, with the Sobolev imbedding \eqref{ineq:Sobolev} of the Appendix,
\[
   \Norm{\sqrt[4]{\ie}}_{L^4(\calT^d)}^2 \leq C_S \Norm{\grad\sqrt[4]{\ie}}_{L^2(\calT^d)}^2 + \Norm{\sqrt[4]{\ie}}_{L^2(\calT^d)}^2
     = \frac{C_S}{4c} \calP_\ie + \int_{\calT^d}\sqrt{\ie}\d x.
\]
This implies  the bound for the last term of $\calH(n,p,\ie)$ in \eqref{F.ind.0},
\(   \label{est_ie}
   \frac{c}{2\sqrt{\ie^*}}\int_{\calT^d}\big(\sqrt{\ie}-\sqrt{\ie^{*}}\big)^2 \dd x \leq \frac{C_S}{4} \calP_\ie.
\)
For bounding the remaining two terms of $\calH(n,p,\ie)$ we first apply the auxiliary Lemma \eqref{aux.1} of the Appendix to get
\begin{equation}\label{bound.n.bar}
    n^* \LB\big(\frac{\bar n}{n^*}\big)
  +  p^* \LB\big(\frac{\bar p}{p^*}\big)
  \leq C_0(n^*,p^*,\bar n, \bar p) \big(\sqrt{\tfrac{\bar n\bar
      p}{n^*p^*}}-1\big)^2, 
\end{equation}
where $C_0$ satisfies \eqref{C0} and is uniformly bounded. 
Now the idea is to bound this right hand side further, where we
use the dissipation term $\calP_R$ resulting from the reactive terms.
We therefore employ the elementary inequality 
\begin{equation}\label{el.in}
\ln(y)(y-1) \geq 4(\sqrt{y}-1)^2\end{equation}
and Jensen's inequality,
\(
   \calP_R &=& \int_{\calT^d} k\, \ln \frac{np}{\ie}\left(\frac{np}{\ie}-1\right) \ie\,\d x
    \geq 4 k_0 \int_{\calT^d} \left(\sqrt{\tfrac{np}{\ie}}-1\right)^2 \ie\,\d x
    \nonumber\\
    &\geq&  2 k_0 \left( \int_{\calT^d} \sqrt{np}-\sqrt{\ie^*} \d x\right)^2
      - 4 k_0 \int_{\calT^d} \left(\sqrt{e}-\sqrt{\ie^*} \right)^2 \d x
\nonumber\\
&\geq & 2 k_0 \left( \int_{\calT^d} \sqrt{np}-\sqrt{\ie^*} \d x\right)^2
      - \tfrac{2 C_S k_0\sqrt{\ie^*}}{c} {\cal P}_e .
   \label{est3}
\)
Inspired by \cite{DiF-Fel-Mar}, we define
\begin{equation}\label{delta.n}
   \delta_n := \sqrt{n} - \int_{\calT^d} \sqrt{n} \d x,\qquad
   \delta_p := \sqrt{p} - \int_{\calT^d} \sqrt{p} \d x,
\end{equation}
and obtain
\begin{align*}
    \calP_R + \frac{2C_S k_0 \sqrt{\ie^*}}{c} \calP_\ie &\geq
        2 k_0 \left( \int_{\calT^d} \sqrt{n} \d x \int_{\calT^d} \sqrt{p} \d x + \int_{\calT^d} \delta_n\delta_p \d x - \sqrt{\ie^*} \right)^2 \\
        &\geq
        k_0  \left( \int_{\calT^d} \sqrt{n} \d x \int_{\calT^d} \sqrt{p} \d x - \sqrt{\ie^*} \right)^2 -
         \frac{ k_0}{2} \left( \int_{\calT^d} (\delta_n^2+\delta_p^2 )\d x \right)^2\\
&\geq  k_0  \left( \int_{\calT^d} \sqrt{n} \d x \int_{\calT^d} \sqrt{p} \d x - \sqrt{\ie^*} \right)^2 -
          \frac{k_0}{2}(\bar n +\bar p)  \int_{\calT^d} (\delta_n^2 + \delta_p^2) \d x \,,
\end{align*}
where we used
\begin{equation}\label{bound.delta.n.1}
 \int_{\calT^d} \delta_n^2 \d x =\ \int_{\calT^d} n \d x -  \Big(\int_{\calT^d} \sqrt{n} \d x\Big)^2 \leq \bar n\,
\end{equation}
and its analog for $p$. 
We can now further apply the Poincar\'e inequality 
\begin{equation}\label{Poinc.delta}
 \int_{\calT^d} \delta_n^2 \d x\leq C_P  \int_{\calT^d} \big|\nabla\sqrt{n}\big|^2 \d x \leq \frac{C_P}{2}\calP_n,
\end{equation}
implying 
\[
    \calP_R + \frac{2C_S k_0 \sqrt{\ie^*}}{c} \calP_\ie &\geq& k_0  \left( \int_{\calT^d} \sqrt{n} \d x \int_{\calT^d} \sqrt{p} \d x - \sqrt{\ie^*} \right)^2 -
          \frac{k_0C_P}{4}(\bar n +\bar p) \big( \calP_n {+} \calP_p\big).
\]
We are therefore left to ``interchange square roots and integration'' in order to connect with \eqref{bound.n.bar}
and complete the estimation of $\calH(n,p,\ie)$ in terms of the entropy production.
We rewrite
\begin{equation}\label{R.n}
   \int_{\calT^d} \sqrt{n} \d x = \sqrt{\bar n} - R_n \int_{\calT^d} \delta_n^2 \d x\qquad \textnormal{with} \quad R_n := \Big( \sqrt{\bar n} + \int_{\calT^d} \sqrt{n} \d x \Big)^{-1}\,.
\end{equation}
Note that $R_n$
is unbounded if and only if $\sqrt{\bar n} \geq \int_{\calT^d} \sqrt{n} \d x$ vanishes. We circumvent here the procedure of distinguishing between different cases as in \cite{DiF-Fel-Mar} by estimating more directly
\begin{eqnarray*}
&&\left( \int_{\calT^d} \sqrt{n} \d x \int_{\calT^d} \sqrt{p} \d x - \sqrt{\ie^*} \right)^2\\
 &&\quad\geq\frac12 \big(\sqrt{\bar n \bar p}-\sqrt{e^*}\big)^2 -\Big(R_n \int_{\calT^d} \delta_n^2 \d x \sqrt{\bar p}+ R_p \int_{\calT^d} \delta_p^2 \d x \sqrt{\bar n} - R_nR_p\int_{\calT^d} \delta_n^2 \d x\int_{\calT^d} \delta_p^2 \d x \Big)^2\\
&&\quad \geq  \frac12 \big(\sqrt{\bar n \bar p}-\sqrt{e^*}\big)^2 - 2 C_P(\bar n+\bar p) (\calP_n + \calP_p)
\end{eqnarray*}
The derivation of the last inequality is carried out in Lemma
\ref{aux.2} of the Appendix giving
\begin{align} \label{C_1}
 & n^* \LB\big(\frac{\bar n}{n^*}\big)
  +  p^* \LB\big(\frac{\bar p}{p^*}\big) \leq C_1(\bar n,\bar
  p,n^*,p^*,\ie^*) \cal P, \quad \text{where} \\
\nonumber 
& C_1(\bar n,\bar p,n^*,p^*,\ie^*)=
   \frac{2}{\ie^*k_0}C_0(\bar n,\bar p,n^*,p^*) \max\Big\{1, \tfrac{2 C_S k_0\sqrt{\ie^*}}{c\kappa},\tfrac{1}{\kappa}\big(\tfrac{k_0}{4}{+}4\big) C_p (\bar n{+}\bar p)\Big\} \,.
\end{align}
Therefore, we can close the chain of inequalities \eqref{F.ind.0}--\eqref{C_1} and conclude
\begin{align}
 & \calH(n,p,\ie) \leq 
   K\calP(n,p,\ie) \quad \text{with} \nonumber \\
 & K(\kappa,\bar n,\bar p,n^*,p^*,\ie^*,\bar\ie,\underline\ie) = 
      \kappa^{-1}\max\Big\{ \kappa C_1(\bar n,\bar p,n^*,p^*,\ie^*), C_{LS}({\ie}/{\ie^*}) , \frac{C_S}{4}\Big\}\,.
      \label{K}
\end{align}
Thus the proof of Proposition \ref{prop:EEP} is complete.
\end{proof}

Using Proposition \ref{prop:EEP}, we are now ready to prove the main
result about exponential convergence towards the steady state.

\begin{theorem}[Exponential convergence to equilibrium]\label{thm:conv1}
Let $(n,p,\ie)$ be a solution to the system \eqref{eq:rdM} subject to the nonnegative initial data
$(n_0,p_0,\ie_0)$ satisfying $\calH(n_0,p_0,\ie_0)$ $ <\infty$, the normalizations \eqref{normalizations} and assumption \eqref{ass.e0}.
Then, the solution converges exponentially fast to the unique constant
equilibrium state $(n^*,p^*,\ie^*)$,
\[
    \Norm{n{-}n^*}_{L^1(\calT^d)}^2 + \Norm{p{-}p^*}_{L^1(\calT^d)}^2
    +  \|\sqrt{\ie}{-}\sqrt{\ie^*}\|_{L^2(\calT^d)}^2 \leq C(\bar n ,\bar p, n^*,p^*) \calH(n_0,p_0,\ie_0) \exp(-\widehat K t),
\]
where $\widehat K$ is given by \eqref{K.max} and  
\[ 
&C(\bar n ,\bar p, n^*,p^*) =\max\Big\{\frac23\big(2(\bar n {+} \bar
p) + 4(C_n {+} C_p)\|\sqrt{e}\|_{L^1(\calT^d)}\big)\,,\,
\tfrac{2\sqrt{\ie^*}}{c+C_n+C_p}(1 {+} 2C_n^2{+}2C_P^2)\Big\}
\]
is uniformly bounded. 
\end{theorem}
\begin{proof} The entropy entropy-production inequality \eqref{entropy-entropy_dis} from Proposition \ref{prop:EEP}
together with the dissipation relation \eqref{diss.rel} imply the exponential convergence of the relative entropy
with exponent $\widehat K$ in \eqref{K.max},
\begin{equation}\label{dec.F}
\calH(n,p,\ie)\leq \calH(n_0,p_0,\ie_0)\exp(-\widehat K t)\,. 
\end{equation}
It therefore remains to derive the decay estimates for $n,p,\ie$ by bounding the relative entropy from below.
We first note that 
\begin{eqnarray*}
    \Norm{n-n^*}_{L^1(\calT^d)}^2 &\leq& 2 \left( \Norm{n-C_n\sqrt\ie}_{L^1(\calT^d)}^2 + C_n^2\|\sqrt\ie-\sqrt{\ie^*}\|_{L^1(\calT^d)}^2 \right)\\
&\leq & 2\left( \Norm{n-C_n\sqrt\ie}_{L^1(\calT^d)}^2 + C_n^2\|\sqrt\ie-\sqrt{\ie^*}\|_{L^2(\calT^d)}^2 \right)\,,
\end{eqnarray*}
and use the Csisz\'ar-Kullback-Pinsker inequality \eqref{ineq:CKP} of
the Appendix, 
\begin{align*}
    \Norm{n-C_n\sqrt\ie}_{L^1(\calT^d)}^2 &\leq \frac13\left( 2 \Norm{n}_{L^1(\calT^d)} + 4 C_n\Norm{\sqrt\ie}_{L^1(\calT^d)} \right)
        \int_{\calT^d} C_n\sqrt\ie\,\LB\big( \tfrac{n}{C_n\sqrt\ie}\big) \,\d x \\
        &\leq \frac13\left( 2 \bar n + 4 C_n \Norm{\sqrt\ie}_{L^1(\calT^d)} \right)\calH(n,p,\ie).
\end{align*}
The analogous estimate holds for $p$.
The proof is concluded by observing that the term $\|\sqrt{\ie}-\sqrt{\ie^*}\|^2_{L^2(\calT^d)}$ is a component of $\calH(n,p,\ie)$
and using the decay \eqref{dec.F}. 
\end{proof}


\subsection{Global existence and long-time behavior in the $x$-dependent case}\label{sec.ex.x}

\subsubsection{Steady states, relative entropy and evolution equations}

As in Section \ref{sec.x.dep.0} we now consider the case of potentials being involved in the dynamics.
In the spirit of semiconductor modeling, the potentials are given by
\[
    V_n = V_\mathrm{conf} + V_\mathrm{el},\qquad V_p = V_\mathrm{conf} - V_\mathrm{el}
\]
where $V_\mathrm{conf}$ represents confinement of the carriers
and $V_\mathrm{el}$ is the electrostatic potential.
However, due to technical difficulties in the derivation of an entropy
entropy-production inequality,
we only study a simplified model with $V_\mathrm{el} \equiv 0$ here,
i.e., we set $V_n=V_p=V:\mathbb{R}^d\rightarrow \mathbb{R}$ $(d\leq 4)$ on the whole space $\mathbb{R}^d$.
In particular, we assume that 
\begin{equation}
\label{V.ass}
V\in C^2(\mathbb{R}^d)\quad \textnormal{ is an}\quad  L^\infty\textnormal{-perturbation of a uniformly convex function.}
\end{equation}

In this section we shall not keep as close track of the constants
arising as in Section \ref{ss:x-indep}.  Moreover, instead of working
on a torus, we consider the full space setting on $\R^d$.  Note that
the presence of the confining potential can be loosely interpreted as
a model for a bounded domain.  

We start from an entropy relation \eqref{eq:whS} but choose $s$,
$w_n$, and $w_p$ depending on $x$. To simplify notation and
without loss of generality we normalize 
\[
\int_{\mathbb{R}^d} \ie \, \dd x =\int_{\mathbb{R}^d} \ie_0 \, \dd x
=1.
\]
As in \cite[Eqn.\,(5.5)]{MiHaMa15UDER} we choose $w_n$ and
$w_p$ linearly dependent in the form
\[
w_n(x,\ie)=C_n \sqrt{\ie}\exp(-V), \quad w_p(x,\ie)=C_p
\sqrt{\ie}\exp(-V)
\]
for some constants $C_n,C_p>0$. Moreover, for $c>0$ we let
\[
\wh s(x,\ie) =c \sqrt{\ie}\exp(-V) \quad \text{giving }
s(x,\ie)= (c{+}C_n{+}C_p)\sqrt e \exp(-V(x)). 
\]
Throughout, we assume the potential to be normalized such that
\[
\int_{\mathbb{R}^d}\exp(-2V) \dd x =1\,.
\] 
The steady state is again determined by maximizing the entropy $\wh S$
under the conservation laws \eqref{cons.laws.0}.
Setting the Lagrangian multiplier for the conservation law of $n-p$ to $0$, i.e. $\Sigma_0=0$, we get
\begin{eqnarray*}
n^*(x)&=&w_n(x,\ie^*(x))=C_n\exp(-V(x))\sqrt{\ie^*(x)}, \\
p^*(x)&=&w_p(x,\ie^*(x))=C_p\exp(-V(x))\sqrt{\ie^*(x)}. 
\end{eqnarray*}
For the Lagrangian multiplier $\Sigma_{\ie}$ corresponding to the conservation of
energy we obtain 
\begin{eqnarray*}
\Sigma_{\ie}=\frac{C_n+C_p+c}{2}\ \frac{\exp(-V(x))}{\sqrt{\ie^*}}.
\end{eqnarray*}
Due to the normalizations of $\ie $ and $V$ we have then
\[
e^*(x)=\exp(-2V(x)) \quad \text{for } x\in \R^d\,.
\]
Thus, we arrive at the following expression for the relative entropy $\calH$:
\begin{eqnarray}
 \calH (n,p,\ie)&=&\int_{\mathbb{R}^d}
 w_n(x,\ie)\LB\big(\tfrac{n}{w_n(x,\ie)}\big)  \dd x +
\int_{\mathbb{R}^d}w_p(x,\ie)\LB\big(\tfrac{p}{w_p(x,\ie)}\big)\dd x \nonumber\\
\label{ent.F.x}&&\quad +\tfrac{c+C_n+C_p}{2}
\int_{\mathbb{R}^d}\Big(\sqrt{\ie}-\sqrt{\ie^*(x)}\Big)^2 \dd x. 
\end{eqnarray}
It is important to realize that no $x$-dependent factor
$\mu=\exp({-}V(x))$ shows 
up in the last integral, because $\mu(x)=\sqrt{e^*(x)}$ gives
$s(x,e)=(c{+}C_n{+}C_p)\sqrt{ e \,e^*(x)}$; hence
\begin{align*}
\tfrac{s(x,e^*(x))-\pl_e s(x,e_*(x))(e{-}e^*(x))- s(x,e)}{c+C_n+ C_p}
=e^*(x) -\tfrac12(e{-}e^*(x))- \sqrt{e\,e^*(x)} =\tfrac12\big(\sqrt e -
\sqrt{e^*(x)}\big)^2. 
\end{align*}

The corresponding gradient system takes the form
\begin{subequations}\label{eq:scM}
\(
    \partial_t n &=& \grad\cdot \left( \grad n + 2n\grad V \right) + k(\ie-\rho(x)np),  \label{scM1} \\
    \partial_t p &=& \grad\cdot \left( \grad p + 2p\grad V \right) + k(\ie-\rho(x)np),  \label{scM2} \\
    \partial_t \ie &=& \grad\cdot \left( \grad\ie + 2\ie\grad V \right).          \label{scM3}
\)
\end{subequations}
Let us summarize for completeness,
\[
    &\rho(x) = \exp(2V), \qquad w_n = C_n \exp(-V)\sqrt\ie, \qquad w_p = C_p \exp(-V)\sqrt\ie,&\\
    &\ie^* = \exp(-2V),\qquad n^* = C_n \exp(-2V),\qquad p^* = C_p \exp(-2V).&
\]
The constants $C_n$, $C_p$ are determined by the relations
\[
   C_n C_p = 1,\qquad C_n-C_p=(C_n-C_p)\int_{\R^d} \exp(-2V) \d x = \int_{\R^d} \left(n^*-p^*\right) \d x = \bfC_0.
\]
The  entropy production reads
\begin{align}
\calP(n,p,\ie)&=\int_{\mathbb{R}^d} \Big(n\Big|\nabla \ln \tfrac{n}{w_n}\Big|^2 +
p\Big|\nabla \ln \tfrac{p}{w_p}\Big|^2\Big)\dd x + \int_{\mathbb{R}^d} \tfrac{N}{4}\Big|\nabla \ln \tfrac{\ie}{\ie^*}\Big|^2 \d x \nonumber\\
&\quad+  \int_{\mathbb{R}^d} k \ie \Big(\tfrac{np}{w_nw_p}-1\Big)\ln\tfrac{np}{w_nw_p} \d x\nonumber\\
&=\frac12\int_{\mathbb{R}^d} \Big(n\Big|\nabla \ln
\tfrac{n}{\ie^*}\Big|^2 {+} n\Big|\nabla \ln \tfrac{n}{\ie}\Big|^2 \Big)\d x
+\frac12\int_{\mathbb{R}^d} \Big(p\Big|\nabla \ln \tfrac{p}{\ie^*}\Big|^2 {+}p\Big|\nabla \ln \tfrac{p}{\ie}\Big|^2 \Big)\d x\nonumber\\
\label{ent.diss.1}& \quad 
+(C_n{+}C_p{+}c)\int_{\mathbb{R}^d} \sqrt{\tfrac{\ie}{\ie^*}}\Big|\nabla \ln \sqrt{\tfrac{\ie}{\ie^*}}\Big|^2 \d x
+ \int_{\mathbb{R}^d} k \ie \big(\tfrac{np}{w_nw_p}{-}1\big)\ln\tfrac{np}{w_nw_p} \d x
\end{align}
where $N=n+p+c\sqrt{\ie}\sqrt{\ie^*}$.

\subsubsection{Global existence of solutions}

We consider the system \eqref{eq:scM} posed in the full space $\R^d$,
subject to the initial data 
\[
    n(0,x) = n_0(x)\geq 0,\qquad  p(0,x) = p_0(x)\geq 0,\qquad
    \ie(0,x)=\ie_0(x)\geq 0. 
\]

We first derive an $L^\infty$-bound for the solution $\ie$ of the Fokker-Planck equation \eqref{scM3}, which is decoupled from the evolution of $n,p$.

\begin{lemma} \label{lemma:L^infty-ie}
Assume that the initial datum $\ie_0\in L^1(\R^d) \cap L^\infty(\R^d)$.
Then the solution $\ie$ to \eqref{scM3} satisfies
\(   \label{L^infty-ie}
   \sup_{t\geq 0} \Norm{\ie(t)}_{L^\infty(\R^d)} < \infty.
\)
\end{lemma}
A proof of this Lemma can be found e.g.\ \cite{DiF-Fel-Mar}.
For the sake of completeness we summarize the main arguments here. 
\\
\begin{proof} 
Let us calculate the evolution of the $L^{j+1}$ norm of $\ie$ via integration by parts,
\[
   \tot{}{t} \int_{\R^d} \ie^{j+1} \d x &=& -\frac{4j}{j+1} \int_{\R^d} \left| \grad\ie^\frac{j+1}{2} \right|^2 \d x +
      2j \int_{\R^d} \ie^{j+1} \laplace V \d x \\
    &\leq&   -\frac{4j}{j+1} \int_{\R^d} \left| \grad\ie^\frac{j+1}{2} \right|^2 \d x +
      2j \Norm{\laplace V}_{L^\infty(\R^d)} \int_{\R^d} \ie^{j+1}  \d x.
\]
We now use the classical Nash inequality, see e.g.\ \cite{Nash58, CL93}, valid for all $f\in L^1(\R^d)\cap H^1(\R^d)$,
\[
    \Norm{f}_{L^2(\R^d)}^\frac{d+2}{d} \leq C_d \Norm{f}_{L^1(\R^d)}^{2/d} \Norm{\grad f}_{L^2(\R^d)},
\]
for $f=\ie^\frac{j+1}{2}$, and the Young inequality with the conjugate exponents
$\frac{d+2}{d}$ and $\frac{d+2}{2}$, to obtain
\[
    \int_{\R^d} \ie^{j+1}  \d x \leq 
      \eps \int_{\R^d} \left| \grad\ie^\frac{j+1}{2} \right|^2 \d x + \frac{C}{\eps} \left(\int_{\R^d} e^\frac{j+1}{2} \d x \right)^2.
\]
Consequently, for $\eps$ of the form $\eps = \eps_j = \frac{A}{j}$ with a suitable constant $A>0$, we have
\[
   \tot{}{t} \int_{\R^d} \ie^{j+1} \d x \leq - \eps_j \int_{\R^d} \ie^{j+1}  \d x
     + Cj(j+\eps_j) \sup_{0\leq\tau\leq t} \left( \int_{\R^d} \ie^\frac{j+1}{2} \d x\right)^2.
\]
By an iterative argument (Lemma 4.2 in \cite{DiF-Fel-Mar}), this implies the announced bound \eqref{L^infty-ie}.
\end{proof}

Next, we derive a uniform $L^1$ bound for $n$ and $p$.

\begin{lemma} \label{lemma:L1bdd-x-dep} Let $(n,p,\ie)$ be a solution
  to the system \eqref{eq:scM} subject to the nonnegative initial data
  $(n_0,p_0,\ie_0)$ with finite entropy $\calH(n_0,p_0,\ie_0) <
  \infty$, satisfying the normalizations \eqref{cons.laws}.  Then
  \( \label{bound_L1_np_2} \sup_{t\geq 0} \left(\Norm{n}_{L^1(\R^d)} +
    \Norm{p}_{L^1(\R^d)} \right) < \infty.  \)
\end{lemma}
\begin{proof}
We first bound
\[
    \Norm{n}_{L^1(\R^d)}^2 \leq 2 \left( \Norm{n-w_n}_{L^1(\R^d)}^2 + \Norm{w_n}_{L^1(\R^d)}^2 \right),
\]
and deduce from the Csisz\'ar-Kullback-Pinsker inequality \eqref{ineq:CKP},
\[
    \Norm{n-w_n}_{L^1(\R^d)}^2 &\leq& \frac13 \left( 2\Norm{n}_{L^1(\R^d)} + 4\Norm{w_n}_{L^1(\R^d)} \right)
       \int_{\R^d}\left( n\ln\frac{n}{w_n} - (n-w_n) \right)\, \d x \\
       &\leq& \frac13 \left( 2\Norm{n}_{L^1(\R^d)} + 4\Norm{w_n}_{L^1(\R^d)} \right) \calH(n,p,\ie).
\]
Due to the entropy production $\calH(n,p,\ie) \leq \calH(n_0,p_0,\ie_0)$ and the mass conservation property for $\ie$, we have
\[
   \Norm{w_n}_{L^1(\R^d)} \leq C_n \Norm{\exp(-2V)}_{L^2(\R^d)} \Norm{\sqrt\ie}_{L^2(\R^d)} = C_n.
\]
Therefore,
\[
   \Norm{n}_{L^1(\R^d)}^2 \leq C_1 \Norm{n}_{L^1(\R^d)} + C_2
\]
for some constants $C_1$, $C_2>0$, which immediately implies the claim for $n$.
Repeating the same steps for $p$, we conclude.
\end{proof}

Finally, we derive uniform $L^\infty$-bounds for $n$ and $p$.

\begin{lemma}
Assume that the initial data $(n_0,p_0,\ie_0)$ are in $(L^1(\R^d) \cap L^\infty(\R^d))^3$
with finite entropy $\calH(n_0,p_0,\ie_0) < \infty$.
Then the solution $(n,p,\ie)$ to the system \eqref{eq:scM} satisfies
\[
   \sup_{t\geq 0} \left( \Norm{n(t)}_{L^\infty(\R^d)} +
     \Norm{p(t)}_{L^\infty(\R^d)} 
       \right) < \infty.
\]
\end{lemma}
\begin{proof}
We use the same Nash-Moser iteration as in the proof of Lemma \ref{lemma:L^infty-ie},
noticing that the only structural difference between the $\ie$-equation \eqref{scM3}
and the $n$, $p$-equations \eqref{scM1}, \eqref{scM2} is the reaction term.
Thus, we only have to use the additional estimate
\[
   \int_{\R^d} (n^j+p^j) k(\ie - \rho np) \d x &\leq& \Norm{k}_{L^\infty(\R^d)} \Norm{\ie}_{L^\infty(\R^d)} \int_{\R^d} (n^j+p^j) \d x \\
     &\leq&  C j \int_{\R^d} (n^{j+1}+p^{j+1}) \d x + \frac{C}{j},
\]
where we used the interpolation of Lebesgue spaces in the second line
and the uniform boundedness of the $L^1$-norms, to derive
\[
   \tot{}{t} \!\int_{\R^d}\!( n^{j+1} {+} p^{j+1}) \d x \leq - \eps
   \int_{\R^d}\!( n^{j+1}  {+} p^{j+1})  \d x
     + Cj(j{+}\eps) \sup_{0\leq\tau\leq t} \left( \int_{\R^d} \!( n^\frac{j+1}{2} {+} p^\frac{j+1}{2}) \d x\right)^2 + \frac{C}{j}.
\]
Again, Lemma 4.2 of \cite{DiF-Fel-Mar} gives uniform boundedness of $n$ and $p$
in $L^\infty$.
\end{proof}

Similarly to the $x$-independent case we need a comparison principle for $\ie$,
which we obtain with respect to the measure $\ie^*$ as follows.

\begin{lemma} Let $\ie_0 \in L^1(\R^d) \cap L^\infty(\R^d)$ satisfy 
\begin{equation}\label{ass.e0.x}
\underline{\ie}\, \ie^*\leq e_0(x)\leq \bar \ie \, \ie^*, \qquad x\in \mathbb{R}^d
\end{equation}
for some $0<\underline{e}\leq \bar e<\infty$. Then the solution $\ie$ to \eqref{scM3} remains within these bounds for all times,
\begin{equation}\label{max.prin.e}
\underline{\ie}\, \ie^*\leq e(t,x)\leq \bar \ie \, \ie^*, \qquad
\qquad x\in \mathbb{R}^d,\ t>0\,. 
\end{equation}
\end{lemma}
\begin{proof}
Recalling that $\ie^*=\exp(-2V)$ is a probability measure, we introduce
\[f=\ie\, \exp(2V)=\frac{\ie}{\ie^*}\]
and equation \eqref{scM3} becomes
\[
   \exp(-2V) \pl_t f = \nabla\cdot(\exp(-2V)\nabla f)\,.
\]
Multiplication by the negative part of $f-\underline \ie$,
i.e. $(f-\underline \ie)_-$, and integration by parts gives
\begin{eqnarray*}
 \frac12 \frac{d}{dt}\int_{\mathbb{R}^d}[(f-\underline \ie)_-]^2 \exp(-2V)\d x = - \int_{\mathbb{R}^d}|\nabla(f-\underline \ie)_-|^2 \exp(-2V) \d x \leq 0\,.
\end{eqnarray*}
Since $(f-\underline \ie)_-(t=0)=0$, this property is retained for all times. The upper bound follows using the same argumentation. 
\end{proof}

\subsubsection{Convergence to equilibrium}
The proof of convergence of the solution towards the stationary state
relies on the exponential decay of the negative entropy $\calH$.  This
follows, as in Section \ref{ss:CtE:x-indep}, from the entropy
entropy-production inequality. Additional care has to be taken here
due to the $x$-dependent stationary states, which requires to
work with the reference probability measure $\ie^*\d x$.

\begin{proposition}[Entropy entropy-production estimate II]\label{prop:EEP2}
  Let $(n,p,\ie)$ be a nonnegative solution to \eqref{eq:scM} in
  $L^\infty(0,\infty; L^1(\R^d))\cap L^\infty((0,\infty)\times \R^d)$,
  and $\ie$ satisfy the bounds \eqref{max.prin.e}.  Then there
  exists a constant $K>0$ such that
\begin{equation}\label{K.x}
  \calH(n,p,\ie) \leq K\, \calP(n,p,\ie) \,.
\end{equation}
\end{proposition}
\begin{proof}
We recall that $\ie \dd x$ and $\ie^*\d x=\exp(-2V)\d x$ are probability measures
and let
\[
   \bar n = \int_{\R^d} n \dd x = \Norm{n}_{L^1}\,,\qquad   \bar p = \int_{\R^d} p \dd x = \Norm{p}_{L^1}\,.
\]

Using the identity
\[
   \int_{\R^d} n \left(\ln \frac{n}{w_n(\ie)}-1\right) \d x &=& \frac12 \int_{\R^d} \left(n\ln \frac{n}{C_n^2\ie^*} + n\ln \frac{n}{\ie} \right) \d x - \bar{n} \\
    &=& \frac12 \int_{\R^d} \left(n\ln \frac{n}{\ie^* \bar{n}} + n\ln \frac{n}{\ie\bar {n}} \right) \d x
    +\left(\bar{n}\ln \frac{\bar{n}}{C_n}- \bar{n}\right)
\]
and its equivalent for $p$, we rewrite the entropy as follows,
\begin{align*}
   \calH(n,p,\ie) &= \frac{1}{2}\int_{\R^d} \left(n\ln \frac{n}{\ie^*\bar{n}}  + n\ln \frac{n\, }{\bar{n}\,\ie } \right) \d x
     + C_n\LB\big(\frac{\bar{n}}{C_n}\big)\\
   &\quad +  \frac{1}{2} \int_{\R^d} \left(p\ln \frac{p}{\ie^*\bar{p}}
     + p\ln \frac{p}{\bar{p}\,\ie } \right) \d x  
     + C_p\LB\big( \frac{\bar{p}}{C_p}\big) \\
   &\quad + \frac{c}{2} \int_{\R^d} \left(\sqrt{\ie}-\sqrt{\ie^*}\right)^2 \d x,
\end{align*}
where we used the identity
\[
   \int_{\R^d} (w_n + w_p)\d x - (C_n+C_p) = -\frac{C_n+C_p}{2} \int_{\R^d} (\sqrt{e}-\sqrt{e^*})^2 \d x.
\]
We also reformulate the entropy production as $\calP(n,p,\ie)=\calP_n + \calP_p + \calP_\ie + \calP_R$, with
\begin{align*}
 &\calP_n = 2 \int_{\R^d} \left|\grad\sqrt{n/\ie^*}\right|^2 \ie^*
      \d x + 2 \int_{\R^d} \left|\grad\sqrt{n/\ie}\right|^2 \ie \dd x,
   &&
  \calP_\ie = 8 \int_{\R^d} \left| \grad\sqrt[4]{\ie/\ie^*} \right|^2 \ie^* \d x,
\\
    &\calP_p = 2 \int_{\R^d} \left|\grad\sqrt{p/\ie^*}\right|^2 \ie^* \d x  + 2 \int_{\R^d} \left|\grad\sqrt{p/\ie}\right|^2 \ie \dd x,
&&
    \calP_R = \int_{\R^d} \!\! k\left(\frac{\rho np}{\ie}{-}1 \right) \ln\frac{\rho np}{\ie} \ie \dd x.
\end{align*}


The generalized logarithmic Sobolev inequality  \cite{AMTU} with respect to the probability measures $\ie^* \d x$, $\ie \d x$
directly implies the following bound on the first and third term of $\calH(n,p,\ie)$:
\[
   \int_{\R^d} \left(n\ln \frac{n}{\bar n\ie^*} + n \ln\frac{n}{\bar n\,\ie} + p \ln\frac{p}{\bar p\ie^* } + p \ln\frac{p}{\bar p\,\ie}\right) \d x
     \leq C_{LS}(\ie,\ie^*)(\calP_n + \calP_p)\,.
\]
Moreover, the Sobolev embedding \eqref{ineq:Sobolev} gives
\[
   1 = \Norm{\sqrt[4]{\ie/\ie^*}}_{L^4(\ie^* \d x)}^2 \leq C\Norm{\nabla\sqrt[4]{\ie/\ie^*}}_{L^2(\ie^* \d x)}^2
       + \Norm{\sqrt[4]{\ie/\ie^*}}_{L^2(\ie^* \d x)}^2
     = C \calP_\ie + \int_{\R^d} \sqrt{\ie}\sqrt{\ie^*} \d x.
\]
Using again the fact that $\ie \d x$ and $\ie^* \d x$ are probability measures, we obtain
\begin{equation}\label{De.x}
   \calP_\ie \geq C\int_{\R^d} \left(\sqrt{\ie}-\sqrt{\ie^*}\right)^2 \d x.
\end{equation}
For the entropy terms containing the averages we proceed similarly as
in the proof of Proposition \ref{prop:EEP} to obtain
\( 
\label{boundS_1} \calH(n,p,\ie) \leq C(\calP_n + \calP_p +
\calP_\ie) + C(\bar{n},\bar{p})
\left(\sqrt{\bar{n}\bar{p}}-1\right)^2, 
\) 
Now the idea is to bound
the additional terms depending on $\bar{n}, \bar{p}$ using the
entropy-production term $\calP_R$ resulting from the reactive terms.
We employ the elementary inequality $\ln(y)(y-1) \geq 4(\sqrt{y}-1)^2$
and Jensen's inequality, also recalling that $n^*
p^*=e^{-2V}\ie^*=\rho^{-1}\ie^*=(\ie^*)^2$, to obtain 
\begin{align} 
\calP_R &=
\int_{\R^d} k\, \ln
\frac{\rho(x)np}{\ie}\left(\frac{\rho(x)np}{\ie}-1\right) \ie \dd x
\geq 4 k_0 \int_{\R^d} \left(\sqrt{\frac{\rho(x)np}{\ie}}-1\right)^2 \ie \dd x =     \nonumber\\
&= 4 k_0 \int_{\R^d} \left(\sqrt{\frac{np}{n^*
      p^*}}-\sqrt{\frac{\ie}{\ie^*}}\right)^2 \ie^* \d x\geq
4k_0\left(\int_{\R^d}\Big( \sqrt{\frac{n p }{(\ie^*)^2}}-
  \sqrt{\frac{\ie}{\ie^*}}\Big) \ie^* \dd x\right)^2.
   \label{boundS_3}
\end{align}
We are therefore left to interchange square roots and integration in order to complete the estimation of $\calH(n,p,\ie)$
in terms of the entropy production.
We shall proceed as in \cite{DiF-Fel-Mar} and introduce a generalization of $\delta_n$, $\delta_p$ in \eqref{delta.n} as follows
\[
   \sqrt{\frac{n}{\ie^*}} = \int_{\R^d} \sqrt{\frac{n}{\ie^*}} \ie^* \d x +\delta_n, \qquad
   \sqrt{\frac{p}{\ie^*}} = \int_{\R^d} \sqrt{\frac{p}{\ie^*}} \ie^* \d x +\delta_p,
\]
and $\delta_\ie$ as
\[
   \sqrt{\frac{\ie}{\ie^*}} = 1 + \delta_\ie \,.
\]
Note that, by definition, we have
\[
   \int_{\R^d} \delta_n \ie^* \d x=0, \qquad
   \int_{\R^d} \delta_n^2 \ie^* \d x \leq  \bar{n} 
\]
Then, due to the Poincar\'e inequality, we have
\[
   \int_{\R^d} \delta_n^2 \ie^* \d x\leq C \Norm{\nabla \sqrt{n/\ie^*}}^2_{L^2(\ie^* \d x)}
\leq C \calP_n\,.
\]
Clearly, analogous properties hold for $\delta_p$. Moreover, due to \eqref{De.x},
\[
   \frac12\int_{\R^d} \delta_\ie^2 \ie^* \d x = 1-\int_{\R^d} \sqrt{\ie}\sqrt{\ie^*} \d x = \frac12 \int_{\R^d} (\sqrt{\ie^*}-\sqrt{\ie})^2 \d x \leq C \calP_\ie.
\]
We now expand \eqref{boundS_3} as
\(  \label{boundS_4}
  \calP_R &\geq& C\left(\int_{\R^d} \sqrt{\frac{n}{\ie^*}} \ie^* \d x \int_{\R^d} \sqrt{\frac{p}{\ie^*}} \ie^* \d x - 1 
   + \int_{\R^d} \delta_n\delta_p \ie^* \d x+\int_{\R^d} \delta_\ie \ie^* \d x\right)^2  \nonumber \\
   &\geq& C\left(\int_{\R^d} \sqrt{\frac{n}{n^*}} \ie^* \d x\int_{\R^d} \sqrt{\frac{p}{p^*}} \ie^* \d x - 1\right)^2
   - C\int_{\R^d} (\delta_n^2+\delta_p^2+\delta_\ie^2) \ie^* \d x \, \nonumber \\
  &\geq& C\left(\int_{\R^d} \sqrt{\frac{n}{n^*}} \ie^* \d x\int_{\R^d} \sqrt{\frac{p}{p^*}}\ie^* \d x- 1\right)^2
   - C (\calP_n + \calP_p + \calP_\ie) \,.
\)
Introducing $R_n =\big( \sqrt{\bar n}+\int_{\mathbb{R}^d}\sqrt{\frac{n}{\ie^*}}\ie^* \d x\big)^{-1}$
and observing that 
$$
\int_{\mathbb{R}^d}\sqrt{\frac{n}{\ie^*}} \ie^* \d
x = \sqrt{\bar n} - R_n\int_{\mathbb{R}^d}\delta_n^2 \dd x,
$$ 
we proceed as in the proof of Proposition \ref{prop:EEP}, using
Lemmas \ref{aux.1} and \ref{aux.2} of the Appendix, to conclude the
desired estimate 
$   \calH(n,p,\ie) \leq K (\calP_n + \calP_p + \calP_\ie + \calP_R) = K\calP(n,p,\ie)$.
\end{proof}

Using this Lemma we are now able to prove convergence towards the steady state.

\begin{theorem}[Exponential convergence towards steady state]
Let $(n,p,\ie)$ be the nonnegative solution of the system \eqref{eq:scM}
with nonnegative initial data $(n_0,p_0,\ie_0)$
such that the initial entropy $\calH(n_0,p_0,\ie_0)$ is finite and $\ie_0$ satisfies \eqref{ass.e0.x}. 
Then the solution converges exponentially fast to the steady state $(n^*,p^*,\ie^*)$,
\(   \label{conv2}
   \Norm{n-n^*}_{L^1(\R^d)}^2 + \Norm{p-p^*}_{L^1(\R^d)}^2 + \|\sqrt{\ie}-\sqrt{\ie^*}\|_{L^2(\R^d)}^2  \leq  C\,\exp(-Kt)
\)
with $C$ a positive constant and $K>0$ as in \eqref{K.x}.
\end{theorem}
\begin{proof}
We write
\[
    \Norm{n-n^*}_{L^1(\R^d)}^2 \leq 2 \left( \Norm{n-w_n}_{L^1(\R^d)}^2 + \Norm{w_n-n^*}_{L^1(\R^d)}^2 \right),
\]
and use the Csisz\'ar-Kullback-Pinsker inequality \eqref{ineq:CKP} of the Appendix,
\[
    \Norm{n-w_n}_{L^1(\R^d)}^2 &\leq& \frac13\left( 2 \Norm{n}_{L^1(\R^d)} + 4 \Norm{w_n}_{L^1(\R^d)} \right)
        \int_{\R^d} n\ln\frac{n}{w_n} - (n-w_n) \,\d x \,.
\]
Then, we proceed as in the proof of Theorem \ref{thm:conv1}: we
combine the uniform boundedness of $\Norm{n}_{L^1(\R^d)}$ and
$\Norm{w_n}_{L^1(\R^d)}$ provided by Lemma \ref{lemma:L1bdd-x-dep},
the dissipation relation \eqref{diss.rel} and the entropy
entropy-production estimate of Proposition \ref{prop:EEP2} to conclude
\eqref{conv2}.
\end{proof}

\appendix
\section{Appendix}\label{s:App}
A probability measure $d\nu$ satisfies the  logarithmic Sobolev inequality if there exists a constant $C>0$ such that
\(   \label{ineq:logSobolev}
     \int f\ln \frac{f}{\|f\|_{L^1(d\nu)}} d\nu \leq C\|\nabla \sqrt{f}\|^2_{L^2(d\nu)}
\)
for every $f\in L^1(\d\nu)$.
For more details we refer to \cite{AMTU}.

The Csisz\'ar-Kullback inequality, see e.g.\ \cite{UAMT00}, states for
the probability densities $f$ and $g $ that 
\(   \label{ineq:CK}
    \Norm{f-g}^2_{L^1} \leq C\int  g \,\LB\big(\frac{f}{g}\big) \d x \,.
\)
A generalization to the case when $f$, $g$ are not probability measures is provided
by the following Czisz\'ar-Kullback-Pinsker inequality: 
\def\iO{\int_\Omega}
\begin{lemma} \label{CziKull}
Let $\Omega$ be a measurable domain in $\R^{d}$. Let $f, g: \Omega \to \R_+$ be measurable. Then,
\begin{equation}  \label{ineq:CKP}
\iO  g \,\LB\big(\frac{f}{g}\big)\dd x \ge
\frac{3}{2\Norm{f}_{L^1} + 4\Norm{g}_{L^1}} \Norm{f-g}^2_{L^1}.
\end{equation}
\end{lemma}
\begin{proof}
The elementary estimate $3|u-1|^2\le (2u+4)\LB(u)$ for $u\in\R$ (Pinsker) gives
\begin{align*}
\Norm{f-g}_{L^1} =& \iO \left|\frac{f}{g}-1\right|g\,\dd x \le 
\iO \sqrt{2\frac{f}{g}+4}\; \sqrt{\LB\big(\frac{f}{g}\big)} \;\frac{g}{\sqrt{3}}\,\dd x\\
\le&\frac{1}{\sqrt{3}}\sqrt{\iO (2f{+}4g)\dd x}\; 
\sqrt{\iO g \,\LB\big(\frac{f}{g}\big)\dd x},
\end{align*}
where we used the Cauchy-Schwarz inequality.
\end{proof}

\begin{lemma}[Sobolev imbedding]
For dimensions $d\leq 4$ we have for a probability measure $\d\nu$ the Sobolev imbedding
\(   \label{ineq:Sobolev}
   \Norm{f}^2_{L^4(\d\nu)} \leq C\Norm{\nabla f}^2_{L^2(\d\nu)} + \Norm{f}_{L^2(\d\nu)}^2\,.
\) 
\end{lemma}
Note that in the standard formulation of the Sobolev imbedding the constant $C$ would multiply the whole right hand side.
For our sake it is however important to have the coefficient 1 in front of the $L^2$-norm.
\\
\begin{proof}
Due to the ``standard'' Sobolev imbedding and the Poincar\'e inequality,
\[
   \Norm{f-\bar f}_{L^4(\d\nu)}\leq C\Norm{f-\bar f}_{H^1(\d\nu)} \leq C\Norm{\nabla f}_{L^2(\d\nu)}
\]
Then, since $\Norm{\bar f}_{L^4(\d\nu)}=\bar f \leq \Norm{f}_{L^2(\d\nu)}$ for the probability measure $\d\nu$, we have
\[
   \Norm{f}_{L^4(\d\nu)}-\Norm{f}_{L^2(\d\nu)} \leq \Norm{f-\bar f}_{L^4(\d\nu)} + \Norm{\bar f}_{L^4(\d\nu)} - \Norm{f}_{L^2(\d\nu)} \leq \Norm{f-\bar f}_{L^4(\d\nu)}.
\]
\end{proof}

\begin{lemma} For any $y > 0$ the following inequality holds,
\begin{equation}\label{el.in.2}
\LB(y) \leq 2 (1+\left|\ln y\right|) (\sqrt{y}-1)^2.
\end{equation}
\end{lemma}
\begin{proof}
In order to show the nonnegativity of 
\[g(y)=2 (1+|\ln y|) (\sqrt{y}-1)^2 -\LB(y)\]
we distinguish between the cases $y\in (0,1]$ and $y>1$. Note that
$g(1)=0$, hence the statement holds true if $g'(y)\leq 0$ for $y\in(0,1]$
and $g'(y)\geq 0$ for $y>1$.

\underline{Case $y\in(0,1)$:} Differentiation of $g$ in this region gives
\[g'(y)=-\frac{2}{y}(\sqrt{y}-1)^2+2(1-\ln y)\frac{\sqrt{y}-1}{\sqrt{y}}-\ln y.\]
Inequality \eqref{el.in}  implies $-\ln y \geq 4(1-\sqrt{y})/(1+\sqrt{y})$ and we obtain 
\[
&g'(y)\geq-\frac{2}{y}(\sqrt{y}-1)^2 -
2\frac{1-\sqrt{y}}{\sqrt{y}} \:h(y) \text{ with } h(y):=1-\ln y
-\frac{2\sqrt{y}}{\sqrt{y}+1}.
\]
Clearly, the first term is nonpositive. For the second term the
same is true as $h$ is nonnegative, because $h(0)=\infty, h(1)=0$ and $h'(y)\leq 0$ for $y\in(0,1]$.

\underline{Case $y\geq 1$:}  In this case differentiation of $g$ gives
\begin{eqnarray*}
g'(y)&=&\frac{2}{y}(\sqrt{y}-1)^2+2(1+\ln y)\frac{\sqrt{y}-1}{\sqrt{y}}-\ln y\,,
\end{eqnarray*}
which we shall prove to be nonnegative. The only negative contribution
is due to the last term. Note first that  for $y\geq 4 $ we have 
$\ln y\Big(2\frac{\sqrt{y}-1}{\sqrt{y}}-1\Big)\geq 0\,$.
Hence, it remains to investigate the case $y\in[1,4]$. Using the elementary inequality $\ln x\leq x-1$ applied to $x=\sqrt{y}$,
we obtain
\begin{eqnarray*}
g'(y)&\geq&\frac{2}{y}(\sqrt{y}-1)^2+2(1+\ln y)\frac{\sqrt{y}-1}{\sqrt{y}}-2(\sqrt{y}-1)\\
&=&2\frac{\sqrt{y}{-}1}{y}\big(-(\sqrt{y}{-}1)^2+\sqrt{y}\ln y\big) 
\ \geq \ 2\frac{\sqrt{y}{-}1}{y}(\sqrt{y}{-}1)^2\big(-1 + 4
\frac{\sqrt{y}}{y{-}1}\big)\, ,
\end{eqnarray*}
where the second inequality is again due to \eqref{el.in}. We see that $g'(y)\geq 0$ also  for $y\in[1,4]$.
\end{proof}

\begin{lemma}\label{aux.1} Let the assumptions of Proposition
  \ref{prop:EEP} hold and let $\delta_n,\delta_p$ be defined as in
  \eqref{delta.n}, then the following estimate holds
 \begin{eqnarray*}
    n^* \LB\big(\frac{\bar n}{n^*}\big)
  +  p^*  \LB\big( \frac{\bar p}{p^*}\big) 
  \leq C_0(\bar n,\bar p,n^*,p^*)\big(\sqrt{\tfrac{\bar n\bar p}{n^*p^*}}-1\big)^2,
\end{eqnarray*}
where the factor 
\begin{equation}\label{C0}
C_0(\bar n,\bar p,n^*,p^*)=C_1(\bar n,\bar p,n^*,p^*)\Big(p^* + \frac{(p^*)^2}{n^*}+2 \frac{n^*}{\max\{\tfrac{\bar p}{p^*},\tfrac{\bar n}{n^*}\}}\Big)
\end{equation}
is uniformly bounded if $(\bar n,\bar p,n^*,p^*)$ are uniformly bounded, with $C_1$ being explicitly given in \eqref{C1}. 
\end{lemma}
\begin{proof}
Using the elementary inequality \eqref{el.in.2} we obtain
\[
&    n^* \LB\big(\frac{\bar n}{n^*}\big)
  +  p^* \LB\big( \frac{\bar p}{p^*}\big) 
  \leq 
  C_1(\bar n,\bar p,n^*,p^*) \Big[ n^*\left(\sqrt{\tfrac{\bar{n}}{n^*}}-1\right)^2 + p^*\left(\sqrt{\tfrac{\bar{p}}{p^*}}-1\right)^2\Big],
\]
where 
\beq\label{C1}
C_1(\bar n,\bar p,n^*,p^*)=\left\{\begin{array}{cll}
2\max\big\{1+\big|\ln \big(\tfrac{\bar n}{n^*}\big)\big|, 1+\big|\ln \big(\tfrac{\bar p}{p^*}\big)\big|\big\}&& \textnormal{if}\ \tfrac{\bar p}{p^*},\tfrac{\bar n}{n^*}\geq \tfrac14
\\
2\left( 1+\big|\ln \big(\tfrac{\bar p}{p^*}\big)\big|\right) && \textnormal{if}\ \tfrac{\bar p}{p^*}\geq \tfrac14\,, \tfrac{\bar n}{n^*}<\tfrac14\\
2\left( 1+\big|\ln \big(\tfrac{\bar n}{n^*}\big)\big|\right) && \textnormal{if}\ \tfrac{\bar n}{n^*}\geq \tfrac14\,, \tfrac{\bar p}{p^*}<\tfrac14
\end{array}\right.
\eeq
which is uniformly bounded due to  Lemma \ref{lemma:L1bdd-x-indep}. 
We next use the following estimate derived in \cite{DiF-Fel-Mar} under the conservation law for $ n -  p$,
\(
   n^*\big(\sqrt{\tfrac{\bar{n}}{n^*}}-1\big)^2 + p^*\big(\sqrt{\tfrac{\bar{p}}{p^*}}-1\big)^2
   \leq C_2(\bar n,\bar p, n^*, p^*) \big(\sqrt{\tfrac{\bar{n}\bar{p}}{n^* p^*}}-1\big)^2 \,.   \label{est1}
\)
In order to keep track of the dependence of the constants on the parameters we give here an explicit bound 
\begin{equation}\label{C2}
C_2(n^*,p^*,\bar n,\bar p)=p^* + \frac{(p^*)^2}{n^*}+2 \frac{n^*}{\max\{\tfrac{\bar p}{p^*},\tfrac{\bar n}{n^*}\}}\,.
\end{equation}
To see this we first note that the conservation law $\bar n-n^* =\bar p- p^*$ can be reformulated as
\begin{eqnarray}\label{rel.n.p}
n^*\Big(\sqrt{\frac{\bar n}{n^*}}-1\Big)\Big(\sqrt{\frac{\bar n}{n^*}}+1\Big)=p^*\Big(\sqrt{\frac{\bar p}{p^*}}-1\Big)\Big(\sqrt{\frac{\bar p}{p^*}}+1\Big)\,.
\end{eqnarray}
This allows us to rewrite
\begin{align*}
&n^*\left(\sqrt{\tfrac{\bar{n}}{n^*}}-1\right)^2 + p^*\left(\sqrt{\tfrac{\bar{p}}{p^*}}-1\right)^2
 =p^*\left(\sqrt{\tfrac{\bar{p}}{p^*}}-1\right)^2\left(\tfrac{n^*\big(\sqrt{\frac{\bar{n}}{n^*}}-1\big)^2}{p^*\big(\sqrt{\frac{\bar{p}}{p^*}}-1\big)^2}+1\right)\\
&= p^*\left(\sqrt{\tfrac{\bar{p}}{p^*}}-1\right)^2\left(\tfrac{p^*\big(\sqrt{\frac{\bar{p}}{p^*}}+1\big)^2}{n^*\big(\sqrt{\frac{\bar{n}}{n^*}}+1\big)^2}+1\right)\\
&\leq  C_3(n^*,p^*,\bar n,\bar p) \left(\sqrt{\tfrac{\bar{p}}{p^*}}-1\right)^2\left(\tfrac{p^*\big(\sqrt{\frac{\bar{p}}{p^*}}+1\big)}
{n^*\big(\sqrt{\frac{\bar{n}}{n^*}}+1\big)}\sqrt{\tfrac{\bar p}{p^*}}+1\right)^2 = C_3(n^*,p^*,\bar n, \bar p) \big(\sqrt{\tfrac{\bar n\bar p}{n^*p^*}}-1\big)^2,
\end{align*}
where we applied again \eqref{rel.n.p} to see the last equality. The multiplier $C_3(n^*,p^*,\bar n,\bar p)$ satisfies the bound
\[
C_3(n^*,p^*,\bar n,\bar p)\geq p^*\left(\tfrac{p^*\big(\sqrt{\frac{\bar{p}}{p^*}}+1\big)^2}{n^*\big(\sqrt{\frac{\bar{n}}{n^*}}+1\big)^2}+1\right)\Big/ \left(\tfrac{p^*\big(\sqrt{\frac{\bar{p}}{p^*}}+1\big)}{n^*\big(\sqrt{\frac{\bar{n}}{n^*}}+1\big)}\sqrt{\tfrac{\bar p}{p^*}}+1\right)^2\,.
\]
Distinguishing between the cases $\sqrt{\frac{\bar p}{p^*}}+1\geq(\leq)\sqrt{\frac{\bar n}{n^*}}+1$, we see that the choice
\[C_3(n^*,p^*,\bar n,\bar p)=p^* + \frac{(p^*)^2}{n^*}+2 \frac{n^*}{\max\{\tfrac{\bar p}{p^*},\tfrac{\bar n}{n^*}\}}\,\]
is sufficient and moreover uniformly bounded. 
\end{proof}

\begin{lemma}\label{aux.2} Let the assumptions of Proposition
  \ref{prop:EEP} hold and let $\delta_n,\delta_p$ and
  $R_n,R_p$ be defined as in \eqref{delta.n} and \eqref{R.n}
  accordingly. Then the estimate
 \begin{eqnarray*}
   \Big(R_n \int_{\calT^d} \delta_n^2 \d x \sqrt{\bar p}+ R_p \int_{\calT^d} \delta_p^2 \d x \sqrt{\bar n} - R_nR_p\int_{\calT^d} \delta_n^2 \d x\int_{\calT^d} \delta_p^2 \d x \Big)^2\leq 2 C_P(\bar n+\bar p) (\calP_n + \calP_p)\,
\end{eqnarray*}
holds.
\end{lemma}
\begin{proof}
We first note that 
\begin{align}
  &\Big(R_n \int_{\calT^d} \delta_n^2 \d x \sqrt{\bar p}+ R_p \int_{\calT^d} \delta_p^2 \d x \sqrt{\bar n} - R_nR_p\int_{\calT^d} \delta_n^2 \d x\int_{\calT^d} \delta_p^2 \d x \Big)^2\nonumber\\
\label{aux.max}& \leq \max \Big\{ \Big(R_n \int_{\calT^d} \delta_n^2 \d x \sqrt{\bar p}+ R_p \int_{\calT^d} \delta_p^2 \d x \sqrt{\bar n} \Big)^2, \Big( R_nR_p\int_{\calT^d} \delta_n^2 \d x\int_{\calT^d} \delta_p^2 \d x \Big)^2\Big\}\,
\end{align}
due to the nonnegativity of both terms.  For bounding these terms we
will make use of the fact that due to the definition of $R_n$ and
$\delta_n$ and the bound in \eqref{bound.delta.n.1} we have
\[
R_n \int_{\calT^d} \delta_n^2 \d x \leq \sqrt{\bar n} \, \qquad \textnormal{and} \qquad R^2_n \int_{\calT^d} \delta_n^2 \d x \leq \tfrac{\sqrt{\bar n} }{\sqrt{\bar n}+\int_{\calT^d} \sqrt{n} \d x}\leq 1\,.
\]
Using these estimates we can proceed as follows 
\begin{eqnarray*}
&&\Big(R_n \int_{\calT^d} \delta_n^2 \d x \sqrt{\bar p}+ R_p \int_{\calT^d} \delta_p^2 \d x \sqrt{\bar n} \Big)^2
\leq  2\Big(R_n^2 \Big(\int_{\calT^d} \delta_n^2 \d x\Big)^2 \bar p+ R_p^2 \Big(\int_{\calT^d} \delta_p^2 \d x\Big)^2 \bar n \Big)\\
&&\quad\leq  2\Big(\bar p \int_{\calT^d} \delta_n^2 \d x + \bar n \int_{\calT^d} \delta_p^2 \d x\Big) \leq 2 (\bar n+\bar p)\int_{\calT^d} (\delta_n^2 +\delta_p^2)\d x \,.
\end{eqnarray*}
The second term in \eqref{aux.max} we now bound by
\begin{eqnarray*}
&&R_n^2R_p^2\Big( \int_{\calT^d} \delta_n^2 \d x\int_{\calT^d} \delta_p^2 \d x \Big)^2 \\
&&= \frac12 R_n^2 \Big(\int_{\calT^d}\delta_n^2 \d x\Big)^2R_p^2 \int_{\calT^d}\delta_p^2 \d x  \, \int_{\calT^d}\delta_p^2 \d x + \frac12R_n^2 \int_{\calT^d}\delta_n^2 \d x R_p^2 \Big(\int_{\calT^d}\delta_p^2 \d x\Big)^2\int_{\calT^d}\delta_n^2 \d x\\
&& \leq \frac{1}{2}\Big(\bar n \int_{\calT^d}\delta_p^2 \d x  + \bar p \int_{\calT^d}\delta_n^2 \d x\Big) \leq 
\frac{1}{2}(\bar n +\bar p)\int_{\calT^d}(\delta_n^2 +\delta_p^2)\d x 
\end{eqnarray*}
Applying finally the Poincar\'e estimate as in \eqref{Poinc.delta} completes the proof. 
\end{proof}

\noindent\textbf{Acknowledgment.} JH and PM are funded by KAUST baseline funds and grant no. 1000000193.
AM was partially supported by Einstein-Stiftung Berlin through the \textsc{Matheon}-Project OT1. SH acknowledges support by the Austrian Science Fund via  
the Hertha-Firnberg project T-764, and the previous funding by the Austrian
Academy of Sciences \"OAW via the New Frontiers project NST-000.

\renewcommand{\baselinestretch}{0.95}
\footnotesize

\bibliographystyle{my_alpha}
\bibliography{alex_pub,bib_alex,HasHitMarMie}

\end{document}